\documentclass[a4paper,11pt,oneside]{article}
\usepackage[bookmarks=false]{hyperref}
\usepackage{amsmath,titlesec,amsthm,amssymb,mathrsfs}
\usepackage{tikz-cd}
\usepackage{tikz}
\usetikzlibrary{matrix,arrows,decorations.pathmorphing}

\titleformat{\section}{\large\bf\boldmath}{\arabic{section}.}{2ex}{}

\titlespacing*{\section}{0ex}{2ex}{1ex}

\titleformat{\subsection}{\bf\boldmath}{\arabic{section}.\arabic{subsection}.}{2ex}{}

\titlespacing*{\subsection}{0ex}{0.8ex}{0ex}

\makeatletter

\newcommand{\Rmnum}[1]{\expandafter\@slowromancap\romannumeral #1@}
\makeatother

\newtheorem{thmstar}{Theorem}
\newtheorem{theorem}{Theorem}[section]
\newtheorem{lemma}[theorem]{Lemma}
\newtheorem{proposition}[theorem]{Proposition}

\newtheorem{corstar}[thmstar]{Corollary}

{\theoremstyle{definition}
\newtheorem{definition}[theorem]{Definition}
\newtheorem*{remark}{Remark}
\newtheorem*{examples*}{Examples}
\newtheorem*{acknowledgements}{Acknowledgements}
}

{\theoremstyle{definition}

}

{\theoremstyle{definition}

}

\newcommand{\B}{\operatorname{B}}

\newcommand{\cF}{\mathcal{F}}
\newcommand{\cC}{\mathcal{C}}

\newcommand{\cH}{\mathcal{H}}
\newcommand{\cG}{\mathcal{G}}
\newcommand{\cK}{\mathcal{K}}
\newcommand{\cL}{\mathcal{L}}

\newcommand{\N}{\mathbb{N}}
\newcommand{\cU}{\mathcal{U}}
\newcommand{\Ker}{\operatorname{Ker}}

\newcommand{\cJ}{\mathcal{J}}
\newcommand{\id}{\mathord{\text{\rm id}}}
\newcommand{\Z}{\mathbb{Z}}
\newcommand{\actson}{\curvearrowright}
\newcommand{\Aut}{\operatorname{Aut}}

\newcommand{\C}{\mathbb{C}}

\newcommand{\Tr}{\operatorname{Tr}}

\newcommand{\lspan}{\operatorname{span}}
\newcommand{\ot}{\otimes}
\newcommand{\bim}[3]{\mathord{\raisebox{-0.4ex}[0ex][0ex]{\scriptsize $#1$}{#2}\hspace{-0.05ex}\raisebox{-0.4ex}[0ex][0ex]{\scriptsize $#3$}}}
\newcommand{\eps}{\varepsilon}
\newcommand{\HNN}{\operatorname{HNN}}
\newcommand{\T}{\mathbb{T}}

\newcommand{\F}{\mathbb{F}}

\newcommand{\dpr}{^{\prime\prime}}

\newcommand{\BS}{\operatorname{BS}}

\newcommand{\op}{^{\text{\rm op}}}
\newcommand{\QN}{\operatorname{QN}}
\newcommand{\Cent}{\operatorname{C}}
\newcommand{\QC}{\operatorname{QC}}

\newcommand{\ovt}{\mathbin{\overline{\otimes}}}

\newcommand{\cS}{\mathcal{S}}

\newcommand{\Ad}{\operatorname{Ad}}
\newcommand{\Gammahat}{\hat{\Gamma}}

\newcommand{\ntil}{\tilde{n}}
\newcommand{\lcm}{\operatorname{lcm}}
\newcommand{\Fix}{\operatorname{Fix}}

\newcommand{\Bim}{\operatorname{Bimod}}

\begin{document}\begin{center}
{\LARGE\bf\boldmath A rigidity result for crossed products of actions of Baumslag-Solitar groups}
\vspace{1ex}

by Niels Meesschaert\footnote{KU~Leuven, Department of Mathematics, Leuven (Belgium), niels.meesschaert@outlook.be \\
    Supported by Research Programme G.0639.11 of the Research Foundation --
    Flanders (FWO).}
\end{center}

\begin{abstract}\noindent
Let $\BS(n_1,m_1)\actson X_1$ and $\BS(n_2,m_2)\actson X_2$ be two ergodic essentially free probability measure preserving actions of nonamenable Baumslag-Solitar groups whose canonical almost normal abelian subgroups act aperiodically. We prove that an isomorphism between the corresponding crossed product II$_1$ factors forces $\BS(n_1,m_1)\cong \BS(n_2,m_2)$ when $|n_1|\neq|m_1|$ and $\BS(n_1,m_1)\cong \BS(n_2,\pm m_2)$ when $|n_1|=|m_1|$. This improves an orbit equivalence rigidity result obtained by Houdayer and Raum in \cite{HR13}.
\end{abstract}

\section*{Introduction}

The \emph{crossed product construction} was introduced by Murray and von Neumann in \cite{MvN36}. It associates to a probability measure preserving (\emph{pmp}) action $\Gamma\actson (X,\mu)$ of a countable group $\Gamma$ on a standard probability space $(X,\mu)$ a von Neumann algebra $L^\infty(X)\rtimes \Gamma$. One of the fundamental problems in this subject is to decide what the crossed product $L^\infty(X)\rtimes \Gamma$ ``remembers'' from the action $\Gamma\actson (X,\mu)$.

For instance, as a consequence of Connes's uniqueness theorem of injective II$_1$ factors (\cite{Co75}), the crossed product $L^\infty(X)\rtimes \Gamma$ of an ergodic essentially free pmp action of an infinite amenable group $\Gamma$ is isomorphic with the unique hyperfinite II$_1$ factor $R$. When $\Gamma$ is nonamenable, the underlying action can sometimes be retrieved completely or partially from the crossed product. For example, in \cite{PV11} it was shown that an isomorphism $L^\infty(X)\rtimes \F_n\cong L^\infty(Y)\rtimes \F_m$, arising from arbitrary ergodic essentially free pmp actions forces $n=m$. We refer to \cite{Io12b} for more on classification of crossed product von Neumann algebras.

For all $n,m \in \Z\setminus\{0\}$, the \emph{Baumslag-Solitar group} $\BS(n,m)$ is defined by the presentation
$$\BS(n,m) := \langle a,b \mid b a^n b^{-1} = a^m \rangle \; .$$
These groups were introduced by Baumslag and Solitar in \cite{BS62} to provide the first examples of two generator non-Hopfian groups with a single defining relation. In \cite{Mo91}, it was shown that $\BS(n_1,m_1)\cong \BS(n_2,m_2)$ if and only if $\{n_1,m_1\}=\{\eps n_2,\eps m_2\}$ for some $\eps\in\{-1,1\}$. Hence we may always assume that $1\leq n \leq |m|$. On the other hand, $\BS(n,m)$ is amenable if and only if $|n| = 1$ or $|m|=1$. Therefore, the groups $\BS(n,m)$ satisfying $2 \leq n \leq |m|$ form a complete list of all nonamenable Baumslag-Solitar groups up to isomorphism. In \cite{MV13}, a partial classification result for the group von Neumann algebras of nonamenable Baumslag-Solitar groups was obtained.

Measure equivalence of countable groups was introduced by Gromov in \cite{Gr93}. Two countable discrete groups $\Gamma$ and $\Lambda$ are called \emph{measure equivalent} if there exist ergodic essentially free pmp actions $\Gamma\actson (X,\mu)$ and $\Lambda\actson (Y,\eta)$ that are stably orbit equivalent. By the work of Ornstein and Weiss in \cite{OW80}, it is known that all infinite amenable groups are measure equivalent to each other. On the other hand, it is well known (c.f. \cite[Proposition 4.3.3]{Zi84}) that nonamenable groups are not measure equivalent to amenable ones. Therefore the measure equivalence class of $\Z$ is exactly the class of all infinite amenable groups. In the nonamenable case, it is in general very hard to determine whether two nonisomorphic groups are measure equivalent. 

Although the question asking whether two nonisomorphic nonamenable Baumslag-Solitar groups are measure equivalent is still open, Kida obtained a measure equivalence rigidity result for nonamenable Baumslag-Solitar groups in \cite{Ki11}. Recently, that result was generalized by Houdayer and Raum (see \cite[Theorem A]{HR13}). It goes as follows in the case $k=l=1$. Let $n_1,m_1,n_2,m_2\in\Z$ such that $2\leq n_1\leq |m_1|$ and $2\leq n_2\leq |m_2|$. They proved that if $\BS(n_1,m_1)$ and $\BS(n_2,m_2)$ have stably orbit equivalent ergodic essentially free pmp actions such that the canonical abelian almost normal subgroups $\langle a_1 \rangle$ and $\langle a_2 \rangle$ act aperiodically (i.e.\ every finite index subgroup acts ergodically), then
\begin{itemize}
\item $n_1=n_2$ and $m_1=m_2$, if $n_1\neq |m_1|$;
\item $n_1=n_2$ and $|m_1|=|m_2|$, if $n_1 = |m_1|$.
\end{itemize}

This brings us to our main result.

\begin{thmstar}\label{Thm.B}
Let $n_1,m_1,n_2,m_2\in\Z$ such that $2\leq n_1\leq |m_1|$ and $2\leq n_2\leq |m_2|$. For $i\in\{1,2\}$, let $(P_i,\tau_i)$ be a diffuse amenable tracial von Neumann algebra and let $\BS(n_i,m_i)\actson P_i$ be a trace preserving action such that the subalgebra $P_i\rtimes (\langle a_i\rangle \cap g \langle a_i\rangle g^{-1}) \subset P_i\rtimes \BS(n_i,m_i)$ is irreducible for every $g\in\BS(n_i,m_i)$. If the crossed products $P_1\rtimes \BS(n_1,m_1)$ and $P_2\rtimes \BS(n_2,m_2)$ are stably isomorphic, then
\begin{itemize}
\item $n_1=n_2$ and $m_1=m_2$, if $n_1\neq |m_1|$;
\item $n_1=n_2$ and $|m_1|=|m_2|$, if $n_1 = |m_1|$.
\end{itemize}
\end{thmstar}

We have the following corollary that generalizes the result of Houdayer and Raum mentioned above.

\begin{corstar}\label{Cor.B}
Let $n_1,m_1,n_2,m_2\in\Z$ such that $2\leq n_1\leq |m_1|$ and $2\leq n_2\leq |m_2|$. For $i\in\{1,2\}$, let $\BS(n_i,m_i)\actson (X_i,\mu_i)$ be an ergodic essentially free pmp action such that $\langle a_i\rangle\actson X_i$ is aperiodic. If the crossed products $L^\infty(X_1)\rtimes \BS(n_1,m_1)$ and $L^\infty(X_2)\rtimes \BS(n_2,m_2)$ are stably isomorphic, then
\begin{itemize}
\item $n_1=n_2$ and $m_1=m_2$, if $n_1\neq |m_1|$;
\item $n_1=n_2$ and $|m_1|=|m_2|$, if $n_1 = |m_1|$.
\end{itemize}
\end{corstar}
\begin{proof}
Define $M_i:=L^\infty(X_i)\rtimes \BS(n_i,m_i)$ and $N_{i,z}:=L^\infty(X_i)\rtimes \langle a_i^z \rangle$ for every nonzero integer $z$. Since the action $\BS(n_i,m_i)\actson (X_i,\mu_i)$ is essentially free, we have that $L^\infty(X_i)' \cap M_i = L^\infty(X_i)$. Therefore, 
$$N_{i,z}' \cap M_i \subset L^\infty(X_i).$$ 
Hence, for every nonzero integer $z$, the relative commutant of $N_{i,z}$ inside $M_i$ is equal to the algebra of $\langle a_i^z\rangle$-invariant functions in $L^\infty(X_i)$. By the aperiodicity of the action $\langle a_i\rangle\actson X_i$, we get that $N_{i,z}$ is an irreducible subalgebra of $M_i$ for every $z$. Using Theorem \ref{Thm.B} yields the desired result.
\end{proof}

\begin{remark}
It is important to note that whenever the crossed products arising in Corollary \ref{Cor.B} would have a unique Cartan subalgebra up to unitary conjugacy, then Corollary \ref{Cor.B} would immediately follow from the result of Houdayer and Raum using Singer's theorem (\cite{Si55}). To the best of our knowledge, there exist no such uniqueness results for these specific crossed products as of this writing. On the other hand, it is possible to find an example of a crossed product admitting at least two Cartan subalgebras when $\langle a \rangle$ is not acting aperiodically. Also note that Houdayer and Raum actually obtain a new OE invariant for certain actions of groups possessing an almost normal subgroup. It would be very interesting to see whether this can be generalized to the framework of II$_1$ factors.
\end{remark}

Let us give a short outline of the proof of Theorem \ref{Thm.B}. For $i\in\{1,2\}$, set $\Gamma_i:=\BS(n_i,m_i)$, $M_i:=P_i\rtimes \Gamma_i$ and $N_i:=P_i\rtimes\langle a_i\rangle$. Let $p$ be a nonzero projection of $N_1$ and let $\alpha:pM_1p\rightarrow M_2$ be an isomorphism. The key to proving the main theorem is to show that $\alpha(pN_1p)$ and $N_2$ are unitarily conjugate.

Theorem \ref{thm.Tech} below will be playing a crucial role in proving this. It is our main technical result and is heavily inspired by Lemma 8.4 of \cite{IPP08}. It roughly says that the relative commutant in $P\rtimes \BS(n,m)$ of every irreducible finite index subalgebra of $P\rtimes \langle a \rangle$ can be controlled in a good way. Concretely, set $M:=P\rtimes \BS(n,m)$ and $N:=P\rtimes \langle a \rangle$. If $p$ is a nonzero projection of $N$ and $Q\subset pNp$ is an irreducible finite index inclusion, then there exists a unitary $u\in\cU(pMp)$ such that $uQu^*\subset pNp$ and $u(Q'\cap pMp)u^*\subset pNp$.

To obtain the unitary conjugacy of $\alpha(pN_1p)$ and $N_2$ we then start by showing that $N_2\prec_{M_2}\alpha(pN_1p)$ and $\alpha(pN_1p)\prec_{M_2} N_2$. Let us only explain how to obtain the first intertwining, since the second intertwining can be obtained in a similar way. Using a slight adaptation of \cite[Lemma 2.3]{BV12}, it actually suffices to show that
\begin{equation}\label{eq.1}
P_2 \prec_{M_2} \alpha(pN_1p) 
\end{equation}
and
\begin{equation}\label{eq.2}
L(\langle a_2\rangle) \prec_{M_2} \alpha(pN_1p).
\end{equation}
Intertwining (\ref{eq.1}) is an immediate corollary of Theorem 4.1 from \cite{Va13} on normalizers in $\HNN$ extensions of von Neumann algebras. For intertwining (\ref{eq.2}), we do the following. Let $\cC$ be the centralizer of $\langle a_2^{n_2}\rangle < \Gamma_2$. We first show that $L(\cC)$ has no amenable direct summand. Then, by combining Theorem 6.4 from \cite{Io12a} and Proposition 3.1 from \cite{Ue07}, we see that $\alpha^{-1}(L(\cC)'\cap M_2)\prec pN_1p$. It follows that $\alpha^{-1}(L(\langle a_2^{n_2}\rangle))\prec pN_1p$ and since $\langle a_2^{n_2}\rangle < \langle a_2\rangle$ has finite index, also $\alpha^{-1}(L(\langle a_2\rangle))\prec pN_1p$. Applying $\alpha$ to both sides, we find intertwining (\ref{eq.2}).

After that, we use our main technical result to control relative commutants and show that the two-sided intertwining of the algebras $\alpha(pN_1p)$ and $N_2$ actually implies unitary conjugacy.

The rest of the proof of Theorem \ref{Thm.B} can then be outlined as follows. By combining all of the above, we may assume that $\alpha(pN_1p)=N_2$. From this, we get that $\bim{N_2}{L^2(M_2)}{N_2}$ is spanned by the irreducible bimodules $\overline{\lspan}~N_2u_gN_2$, but also by the irreducible bimodules $\overline{\lspan}~\alpha(pN_1u_gN_1p)$. By examining the left and right dimensions of these bimodules, we will get that both $n_1=n_2$ and $|m_1|=|m_2|$. When $n_1\neq |m_1|$, a further careful study of the bimodules $\overline{\lspan}~N_2u_gN_2$ and $\overline{\lspan}~\alpha(pN_1u_gN_1p)$ moreover yields $m_1=m_2$.

\begin{acknowledgements}
The author is very grateful to his advisor Stefaan Vaes for his helpful comments and careful reading of the various drafts of this paper. Also, many thanks go to Sven Raum for suggesting several improvements.
\end{acknowledgements}

\section{Preliminaries}\label{sec.prelim}

We denote by $(M,\tau)$ a von Neumann algebra equipped with a faithful normal tracial state $\tau$. We always assume that $M$ has a \emph{separable predual}. 

Let $N$ be a von Neumann subalgebra of a tracial von Neumann algebra $(M,\tau)$. We denote by $E_N$ the unique trace preserving \emph{conditional expectation} of $M$ onto $N$.

\subsection{Hilbert bimodules}\label{sec.bim}

We start by giving the definition of a Hilbert bimodule.

\begin{definition}[see e.g.\ \cite{Co94}]
Let $(M,\tau_M)$ and $(N,\tau_N)$ be tracial von Neumann algebras.
\begin{enumerate}
\item A \emph{left $M$-module} $_M\cH$ is a Hilbert space $\cH$ equipped with a normal unital $*$-homomorphism $\pi_l:M\rightarrow \B(\cH)$;
\item A \emph{right $N$-module} $\cH_N$ is a Hilbert space $\cH$ equipped with a normal unital $*$-anti-homo-morphism $\pi_r:N\rightarrow \B(\cH)$ (i.e.\ a normal unital representation of the opposite algebra $N\op$);
\item An \emph{$M$-$N$-bimodule} $\bim{M}{\cH}{N}$ is a Hilbert space which is both a left $M$-module and a right $N$-module, such that the representations $\pi_l$ and $\pi_r$ commute.
\end{enumerate}
\end{definition}

Let $\cH$ be an $M$-$N$-bimodule, for $x\in M, y\in N$ and $\xi\in \cH$, we write $x\xi y$ instead of $\pi_l(x)\pi_r(y)(\xi)$.

If $\bim{M}{\cH}{N}$ is an $M$-$N$-bimodule, the \emph{contragredient bimodule} $\bim{N}{\overline{\cH}}{M}$ is defined on the conjugate Hilbert space $\overline{\cH} = \cH^*$ with bimodule actions given
by
\[x\cdot \overline{\xi}=\overline{\xi x^*} \text{ and } \overline{\xi}\cdot y = \overline{y^*\xi}.\]

The following proposition (see e.g.\ \cite[Theorem 2.2.2]{JS97}) is a well known result saying that all right $M$-modules are of a special form. Recall that two projections $p$ and $q$ in a von Neumann algebra are called \emph{equivalent} (\cite{MvN36}), if there exists an element $v$ in that von Neumann algebra satisfying $p=v^*v$ and $q=vv^*$. 

\begin{proposition}\label{lem.generalbim}
Let $(M,\tau)$ be a tracial von Neumann algebra and let $\cH$ be a countably generated right $M$-module. Then there exists a projection $p\in M^\infty:=\B(l^2(\N))\ovt M$, which can be taken diagonal, such that $\cH$ and $p(l^2(\N)\otimes L^2(M))$ are isomorphic as right $M$-modules. Moreover, this correspondence defines a bijection between the class of countably generated right $M$-modules, up to isomorphism, and the set of equivalence classes of projections in $\B(l^2(\N))\ovt M$.
\end{proposition}
%

Define the (infinite) trace 
\[\Tr:\B(l^2(\N))^+\rightarrow [0,+\infty]:x\mapsto\sum_{n\in\N}\langle xe_n,e_n\rangle,\]
where $(e_n)_n$ denotes the canonical orthonormal basis of $l^2(\N)$. Following the notation of the previous proposition, since $(\Tr\otimes \tau_M) (p)$ is an invariant for the equivalence class of $p$, it is also an invariant for the isomorphism class of the right $M$-module $\cH$. This invariant is called the \emph{right dimension} of $\cH$ and is denoted by $\dim_{-M}(\cH)$. By considering a left $M$-module as a right $M\op$-module, we can also define the \emph{left dimension} $\dim_{M-}(\cH)$ of a left $M$-module $\cH$. Moreover, a bimodule $\bim{M}{\cH}{N}$ is said to have \emph{finite index} when the dimension of $_M\cH$ and $\cH_N$ are both finite. Also, we call an $M$-$N$-bimodule \emph{bifinite} if it is finitely generated both as a left Hilbert $M$-module and a right Hilbert $N$-module. Finally, if $(M,\tau)$ is a tracial von Neumann algebra and $N$ is a von Neumann subalgebra of $M$, then we define the \emph{Jones index} $[M:N]$ as $\dim_{\text{-}N}L^2(M)$, see \cite{Jo83}.


\subsection{Popa's intertwining-by-bimodules}\label{sec.inter}

An inclusion of von Neumann algebras $A\subset M$ is called \emph{nonunital} whenever the unit of $A$ does not coincide with the unit of $M$. In that case, we also call $A$ a \emph{nonunital von Neumann subalgebra} of $M$. In the next theorem we allow such nonunital inclusions. The result itself is called Popa's intertwining-by-bimodules theorem.

\begin{theorem}[{\cite[Theorem 2.1 and Corollary 2.3]{Po03}}]\label{Thm.intertwining}
Let $(M,\tau)$ be a tracial von Neumann algebra and let $A,B\subset M$ be possibly nonunital von Neumann subalgebras. Denote their respective units by $1_A$ and $1_B$. The following five conditions are equivalent:
\begin{enumerate}
\item $1_A L^2(M) 1_B$ admits a nonzero $A$-$B$-subbimodule that is finitely generated as a right $B$-module.
\item $1_A L^2(M) 1_B$ admits a nonzero $A$-$B$-subbimodule that has finite right $B$-dimension.
\item There exist nonzero projections $p\in A$, $q\in B$, a normal unital $*$-homomorphism $\psi:pAp\rightarrow qBq$ and a nonzero partial isometry $v\in pMq$ such that $av=v\psi(a)$ for all $a\in pAp$.
\item There exists a nonzero projection $q\in B^n$, a normal unital $*$-homomorphism $\psi:A\rightarrow qB^nq$ and a nonzero partial isometry $v\in (M_{1,n}(\C)\otimes 1_AM)q$ such that $av=v\psi(a)$ for all $a\in A$.
\item There is no sequence of unitaries $u_n\in \cU(A)$ satisfying $||E_B(xu_ny^*)||_2\rightarrow 0$ for all $x,y\in 1_BM1_A$.
\end{enumerate}
If one of the equivalent conditions holds, we write $A \prec_M B$.
\end{theorem}

The following lemma can be found in \cite[Lemma 3.4]{Va08} where the proof is left as an exercise. For the convenience of the reader we provide a proof here.

\begin{lemma}\label{lem.interproj}
Let $A,B\subset (M,\tau)$ be, possibly nonunital, embeddings. Let $q_0\in A$, $q_1\in A'\cap 1_AM1_A$, $p_0\in B$ and $p_1\in B'\cap 1_BM1_B$ be nonzero projections.
\begin{itemize}
\item If $q_0Aq_0\prec_M B$ or if $q_1A\prec_M B$, then $A\prec_M B$.
\item If $A\prec_M p_0Bp_0$ or if $A\prec_M p_1B$, then $A\prec_M B$.
\end{itemize}
\end{lemma}
\begin{proof}
The fact that $q_0Aq_0\prec_M B$ and $A\prec_M p_0Bp_0$ both imply that $A\prec_M B$, follows immediately from the third characterisation in Theorem \ref{Thm.intertwining}. On the other hand, using the first characterisation in Theorem \ref{Thm.intertwining}, we also see that $q_1A\prec_M B$ and $A\prec_M p_1B$ both imply that $A\prec_M B$.
\end{proof}

The following lemma can for instance be found in \cite[Lemma 3.9]{Va08}.

\begin{lemma}\label{lem.interfi}
Let $A,B\subset (M,\tau)$ be, possibly nonunital, embeddings.
\begin{itemize}
\item If $A\prec_M B$ and if $D\subset B$ is a unital finite index inclusion, then $A\prec_M D$;
\item If $A\prec_M B$ and if $A\subset D$ is a unital finite index inclusion, then $D\prec_M B$.
\end{itemize}
\end{lemma}

%

Next to intertwining-by-bimodules, we also need a stronger notion called \emph{full embedding}.

\begin{definition}
Let $(M,\tau)$ be a tracial von Neumann algebra and let $A,B\subset M$ be possibly nonunital von Neumann subalgebras. Denote the unit of $A$ by $1_A$. We say that \emph{$A$ embeds fully into $B$ inside $M$}, and we write $A\prec_M^f B$, if $Ap\prec_M B$ for every nonzero projection $p\in 1_AM1_A\cap A'$.
\end{definition}

One of the advantages that full embedding has over intertwining is that the relation '$\prec_M^f$' is transitive, while the relation '$\prec_M$' need not be. Indeed, let $p$ be a nontrivial projection in a diffuse tracial von Neumann algebra $(M,\tau)$, then we have that $M\prec_M pMp+\C(1-p)$ and $pMp+\C(1-p)\prec_M \C$, but $M\not\prec_M \C$. On the other hand, Lemma \ref{Lem.trans} below implies that the relation '$\prec_M^f$' is always transitive. A proof can be found in \cite[Lemma 3.7]{Va08}.

\begin{lemma}\label{Lem.trans}
Let $(M,\tau)$ be a tracial von Neumann algebra and let $A,B,D\subset (M,\tau)$ be possibly nonunital embeddings. If $A\prec_M B$ and $B\prec_M^f D$, then $A\prec_M D$.
\end{lemma}

A way to obtain full embedding is by quasi-regularity. We first recall the definition of quasi-regularity.

\begin{definition}
Let $(M,\tau)$ be a tracial von Neumann algebra and $N\subset M$ a von Neumann subalgebra. We denote by $\QN_M(N)$ the \emph{quasi-normalizer} of $N$ inside $M$, i.e.\ the unital $*$-algebra defined by
\[\Bigl\{a\in M \;\Big|\; \exists b_1,\ldots, b_k\in M,\exists d_1,\ldots,d_r\in M : \; Na\subset \sum_{i=1}^k b_iN,\; aN\subset \sum_{j=1}^r Nd_j\Bigr\} \; .\]
Whenever $\QN_M(N)\dpr=M$, we call $N\subset M$ \emph{quasi-regular}.
\end{definition}



The following two results are well known to experts.

\begin{lemma}\label{Lem.full}
Let $(M,\tau)$ be a tracial von Neumann algebra and let $A,B\subset M$ be possibly nonunital von Neumann subalgebras. If $A\prec_M B$ and $\QN_{1_AM1_A}(A)\dpr$ is an irreducible von Neumann subalgebra of $1_AM1_A$, then $A\prec_M^f B$.
\end{lemma}
\begin{proof}
Let $p\in A'\cap 1_AM1_A$ be a nonzero projection. We prove that there exists a nonzero $Ap$-$B$-subbimodule $\cK$ of $pL^2(M)1_B$ with finite $B$-dimension.

Since $A\prec_M B$, there exists a nonzero $A$-$B$-subbimodule $\cH$ of $1_AL^2(M)1_B$ with finite $B$-dimension. Since $\QN_{1_AM1_A}(A)\dpr$ is an irreducible von Neumann subalgebra of $1_AM1_A$, we have that $1_AL^2(M)1_B$ is an irreducible $\QN_{1_AM1_A}(A)\dpr$-$(1_B M 1_B)$-bimodule. This then implies that 
$$1_AL^2(M)1_B=\overline{\lspan}(\QN_{1_AM1_A}(A)\dpr\cH M1_B).$$ Therefore $p\QN_{1_AM1_A}(A)\dpr\cH M1_B\neq \{0\}$, or equivalently $p\QN_{1_AM1_A}(A)\dpr\cH\neq \{0\}$. Hence, there exists an element $v\in \QN_{1_AM1_A}(A)$ such that $pv\cH$ is nonzero. Write $\cK:=\overline{\lspan} (pAv\cH)$, then $\cK$ is a nonzero $Ap$-$B$-subbimodule of $pL^2(M)1_B$ with finite $B$-dimension.
\end{proof}

Let $(M,\tau)$ be a tracial von Neumann algebra and let $A,B\subset (M,\tau)$ be von Neumann subalgebras. If $A\prec_M B$ and $B\prec_M A$, then $L^2(M)$ does admit both a nonzero $A$-$B$-subbimodule with finite $B$-dimension and one with finite $A$-dimension. However this does not immediately imply that $L^2(M)$ admits a nonzero finite index $A$-$B$-subbimodule, but we do have the following proposition.

\begin{proposition}\label{prop.full}
Let $(M,\tau)$ be a tracial von Neumann algebra and let $A,B\subset (M,\tau)$ be possibly nonunital von Neumann subalgebras. If $A\prec^f_M B$, $B\prec_M A$ and $B$ is quasi-regular inside $1_BM1_B$, then there exists a nonzero finite index $A$-$B$-subbimodule of $1_AL^2(M)1_B$.
\end{proposition}
\begin{proof}
Since $B\prec_M A$, there exists a nonzero subbimodule $\bim{A}{\cH}{B}$ of $\bim{A}{(1_AL^2(M)1_B)}{B}$ with finite $A$-dimension. Consider the $A$-$1_BM1_B$-subbimodule $\overline{\lspan}(\cH M1_B)$ of $L^2(M)1_B$. There exists a projection $p\in A'\cap 1_AM1_A$ such that $\overline{\lspan}(\cH M1_B)=pL^2(M)1_B$. Since $B$ is quasi-regular inside $1_BM1_B$, we therefore get that $pL^2(M)1_B$ is densely spanned by the $A$-$B$-subbimodules
\[\bim{A}{\overline{\lspan}(\cH v B)}{B} \text{ with } v\in \QN_{1_BM1_B}(B).\]
Now since $A\prec_M^f B$, there also exists a nonzero $A$-$B$-subbimodule $\cK$ of $pL^2(M)1_B$ with finite $B$-dimension. Take an element $v\in \QN_{1_BM1_B}(B)$ such that the orthogonal projection $p_v$ of $pL^2(M)1_B$ onto $\overline{\lspan}(\cH v B)$ satisfies $p_v(\cK)\neq \{0\}$. Then $\bim{A}{(p_v(\cK))}{B}$ is a nonzero subbimodule of $\bim{A}{\overline{\lspan}(\cH v B)}{B}$ with $\dim_{-B}(p_v(\cK))\leq \dim_{-B}(\cK)<\infty$. Since $\dim_{A-}(\overline{\lspan}(\cH v B))<\infty$, also $\dim_{A-}(p_v(\cK))<\infty$. This ends the proof.
\end{proof}

\subsection{Connes tensor products}\label{Sec.Connes}

Let $(N,\tau_N)$ be a tracial von Neumann algebra. Let $\cH$ be a right $N$-module and let $\cK$ be a left $N$-module. We call a vector $\xi\in\cH$ \emph{right bounded}, if there exists $c>0$ such that $||\xi x||\leq c||x||_2$, for all $x\in N$. In that case we define the bounded linear operator $L_\xi:L^2(N)\rightarrow \cK$ as 
\[L_\xi(\hat{x})=\xi x\text{ for every }x\in N,\]
where $\hat{x}$ denotes $x$ when viewed as a vector of $L^2(N)$. Denote by $\cH_0$ the vector space of all right bounded vectors of $\cH$. On the algebraic tensor product $\cH_0 \odot \cK$, we define the inner product
\[\langle \xi\otimes \eta,\xi'\otimes \eta'\rangle = \langle (L_{\xi'}^*L_\xi) \eta,\eta'\rangle.\]
Note that this makes sense, since $L_{\xi'}^*L_\xi$ is a bounded linear operator on $L^2(N)$ that commutes with the right $N$-action, and hence must be an element of $N$. The \emph{Connes tensor product} $\cH\otimes_N \cK$ (see Appendix B.$\delta$ of \cite{Co94}) is then defined as the completion of $(\cH_{0}\odot \cK)/N_{\langle\cdot,\cdot\rangle}$, where $N_{\langle\cdot,\cdot\rangle}:=\{\zeta\in \cH_{0}\odot \cK\mid \langle \zeta,\zeta\rangle = 0\}$.

The Connes tensor product $\cH\otimes_N \cK$ can also be obtained by looking at left bounded vectors of $\cK$. We call a vector $\eta\in\cK$ \emph{left bounded}, if there exists $c>0$ such that $||x\eta||\leq c||x||_2$, for all $x\in N$. In that case we can define a bounded linear operator $R_\eta:L^2(N)\rightarrow \cK$ by 
\[R_\eta(\hat{x})=x\eta\text{ for every }x\in N.\] 
We denote by $_{0}\cK$ the vector space of all left bounded vectors of $\cK$. On the algebraic tensor product $\cH\odot {}_{0}\cK$, we define the inner product
\[\langle \xi\otimes \eta,\xi'\otimes \eta'\rangle = \langle \xi (JR_{\eta}^*R_{\eta'}J),\xi'\rangle,\]
where $J:\hat{x}\mapsto \widehat{x^*}$ is the canonical anti-unitary on $L^2(N)$. The Connes tensor product $\cH\otimes_N \cK$ is equivalently defined as the completion of $(\cH\odot {}_{0}\cK)/N_{\langle\cdot,\cdot\rangle}$, where $N_{\langle\cdot,\cdot\rangle}:=\{\zeta\in \cH\odot {}_{0}\cK\mid \langle \zeta,\zeta\rangle = 0\}$ (see e.g.\ \cite{Po86}).

The following is an important property of the Connes tensor product.

\begin{proposition}
Let $(N,\tau_N)$ be a tracial von Neumann algebra. Let $\cH$ be a right $N$-module and let $\cK$ be a left $N$-module. Inside $\cH\otimes_N \cK$, we have that 
\begin{itemize}
\item $\xi x\otimes \eta=\xi \otimes x\eta$ for every $\xi\in \cH_0$, $\eta\in \cK$ and $x\in N$;
\item $\xi x\otimes \eta=\xi \otimes x\eta$ for every $\xi\in \cH$, $\eta\in {}_{0}\cK$ and $x\in N$.
\end{itemize}
\end{proposition}
\begin{proof}
Let $\xi,\xi'\in \cH_0$, $\eta,\eta'\in \cK$ and $x\in N$, then we have that
\begin{align*}
\langle \xi x \otimes \eta,\xi'\otimes \eta'\rangle 
&= \langle (L_{\xi'}^*L_{\xi x})\eta,\eta'\rangle = \langle (L_{\xi'}^*L_{\xi}x)\eta,\eta'\rangle\\
&= \langle (L_{\xi'}^*L_{\xi})x\eta,\eta'\rangle = \langle \xi \otimes x\eta,\xi'\otimes \eta'\rangle. 
\end{align*}
This shows that $\xi x\otimes \eta=\xi \otimes x\eta$ for every $\xi\in \cH_0$, $\eta\in \cK$ and $x\in N$. A similar argument can be used to prove the second statement.
\end{proof}

Let $(M,\tau_M)$, $(N,\tau_N)$ and $(P,\tau_P)$ be tracial von Neumann algebras. Let $\cH$ be an $M$-$N$-bimodule and let $\cK$ be an $N$-$P$-bimodule. The Hilbert space $\cH\otimes_N \cK$ is then an $M$-$P$-bimodule, where $x(\xi\otimes \eta)y=x\xi\otimes \eta y$ (see \cite[Theorem 13]{Co94}). Whenever $N$ is a II$_1$ factor, the left and right dimension of this bimodule can easily be computed from the left and right dimensions of $\cH$ and $\cK$. This is the content of the following proposition which can be found in \cite{JS97} without proof. For the convenience of the reader, we provide a proof for this result.

\begin{proposition}\label{Prop.dimConnes}
Let $(M,\tau_M)$ and $(P,\tau_P)$ be tracial von Neumann algebras and let $(N,\tau_N)$ be a II$_1$ factor. If $\cH$ is a finite index $M$-$N$-bimodule and $\cK$ is a finite index $N$-$P$-bimodule, then
\begin{itemize}
\item $\dim_{-P}(\cH\otimes_N \cK) = \dim_{-N}(\cH)\dim_{-P}(\cK)$;
\item $\dim_{M-}(\cH\otimes_N \cK) = \dim_{M-}(\cH)\dim_{N-}(\cK)$.
\end{itemize}
\end{proposition}
\begin{proof}
We restrict ourselves to proving the first equality, since the second equality can be proven analogously.

So let $(P,\tau_P)$ be tracial von Neumann algebra and let $(N,\tau_N)$ be a II$_1$ factor. Furthermore, let $\cH$ be a right $N$-module with $\dim_{-N} \cH < \infty$ and let $\cK$ be an $N$-$P$-bimodule with $\dim_{-P}\cK <\infty$. Using Proposition \ref{lem.generalbim}, we may assume that 
\[\cH_N=(\bigoplus_i p_iL^2(N))_N,\]
where $p_i\in N$ are projections satisfying $\dim_{-N}(\cH)=\sum_i \tau_N(p_i)$. In this way, we have that
\[(\cH\otimes_N \cK)_P=(\bigoplus_i p_i\cK)_P.\]
On the other hand, using Proposition \ref{lem.generalbim} again, we may assume that
\[\cK_P=p(l^2(\N)\otimes L^2(P))_P,\]
where $p\in \B(l^2(\N))\ovt P$ is a projection satisfying $\dim_{-P}(\cK)=(\Tr\otimes\tau_P)(p) < \infty$. But since $\cK$ is also a left $N$-module, there exists a normal $*$-homomorphism $\psi:N\rightarrow pP^\infty p$ such that the corresponding left $N$-action on $p(l^2(\N)\otimes L^2(P))$ is given by $x\cdot \xi=\psi(x)\xi$ for every $x\in N$. In this way, we see that
\[\dim_{-P}(p_i\cK)=(\Tr\otimes\tau_P)(\psi(p_i)).\]
Since $N$ is a factor, we have that $\psi$ is faithful (see \cite[Proposition II.3.12]{Tak79}). Hence we have that $(\Tr\otimes\tau_P)(\psi(\cdot))/(\Tr\otimes\tau_P)(p)$ is a normal faithful tracial state on $N$. Since $N$ is a II$_1$ factor, we have that $\tau_N$ is the unique normal faithful tracial state (see \cite[Theorem V.2.6]{Tak79}). Therefore
\[(\Tr\otimes\tau_P)(\psi(x))=\tau_N(x)(\Tr\otimes\tau_P)(p) \text{ for every }x\in N.\]
Putting everything together, we see that
\begin{align*}
\dim_{-P}(\cH\otimes_N \cK) &= \dim_{-P}(\bigoplus_i p_i\cK) = \sum_i \dim_{-P}(p_i\cK)\\
&= \sum_i (\Tr\otimes\tau_P)(\psi(p_i)) = \sum_i \tau_N(p_i)(\Tr\otimes\tau_P)(p)\\
&= \dim_{-N}(\cH)\dim_{-P}(\cK).
\end{align*}
\end{proof}

\subsection{Baumslag-Solitar groups and $\HNN$ extensions}\label{Sec.HNN}

For all $n,m \in \Z \setminus \{0\}$, the Baumslag-Solitar group $\BS(n,m)$ is defined by the presentation
\[\BS(n,m) := \langle a,b \mid b a^n b^{-1} = a^m \rangle \; .\]
The Baumslag-Solitar groups were introduced in \cite{BS62} as the first examples of two generator non-Hopfian groups with a single defining relation. Ever since, they have been playing an important role in many different areas of mathematics. The following are two examples of this.
\begin{itemize}
\item In theoretical informatics they were the first groups that were known to be asynchronous automatic but not automatic (see e.g.\ \cite{ECH+92});
\item In geometric group theory they provide easy examples of groups that are not isomorphic to a subgroup of a hyperbolic group (see e.g.\ \cite{GS91}).
\end{itemize}

Since they play such important roles in mathematics, it is a natural problem to classify von Neumann algebras arising in some way from Baumslag-Solitar groups.

The following facts will be useful later on. Whenever $|n| = 1$ or $|m|=1$, the normal closure of $\langle a \rangle$ is an abelian normal subgroup of $\BS(n,m)$ such that the quotient is infinite cyclic. So in that case, the group $\BS(n,m)$ is solvable, hence amenable. Whenever $|n| \geq 2$ and $|m| \geq 2$, the subgroup $\langle b, aba^{-1}\rangle\leq \BS(n,m)$ is, by Lemma \ref{lem.Brit} below, isomorphic with the free group $\F_2$. So in that case, $\BS(n,m)$ is nonamenable. In \cite{Mo91} it was proven that $\BS(n_1,m_1)\cong \BS(n_2,m_2)$ if and only if $\{n_1,m_1\}=\{\eps n_2,\eps m_2\}$ for some $\eps\in\{-1,1\}$. So all nonamenable Baumslag-Solitar groups are up to isomorphism of the form $\BS(n,m)$ for some $2 \leq n \leq |m|$.

Now, let us introduce the notion of $\HNN$ extension of groups (\cite{HNN49}). Let $G$ be a group, $H < G$ a subgroup and $\theta:H\rightarrow G$ an injective group homomorphism. The \emph{HNN extension} $\HNN(G,H,\theta)$ is defined by the presentation
\[\HNN(G,H,\theta) = \langle G,b \mid \theta(h)=bhb^{-1} \;\;\text{for all} \;\; h\in H\rangle \; .\]

Elements of $\HNN(G,H,\theta)$ can be expressed in a `reduced' way using as letters the elements of $G$ and the letters $b^{\pm 1}$. More precisely, we have the following lemma.

\begin{lemma}[Britton's lemma, \cite{Br63}]\label{lem.Brit}
Consider the expression $g = g_0 b^{n_1} g_1 b^{n_2} \cdots b^{n_k} g_k$ with $k \geq 0$, $g_0,g_k \in G$, $g_1,\ldots,g_{k-1} \in G~\backslash~\{e\}$ and $n_1,\ldots,n_k\in \Z~\backslash~\{0\}$. We call this expression reduced if the following two conditions hold:
\begin{itemize}
\item for every $i \in \{1,\ldots,k-1\}$ with $n_i > 0$ and $n_{i+1} < 0$, we have $g_i \not\in H$;
\item for every $i \in \{1,\ldots,k-1\}$ with $n_i < 0$ and $n_{i+1} > 0$, we have $g_i \not\in \theta(H)$.
\end{itemize}
If the above expression for $g$ is reduced, then $g \neq e$ in the group $\HNN(G,H,\theta)$, unless $k=0$ and $g_0 = e$. In particular, the natural homomorphism of $G$ to $\HNN(G,H,\theta)$ is injective.
\end{lemma}

The number $\sum_{i=1}^k |n_i|$ appearing in a reduced expression of $g$ is called the \emph{$b$-length} of $g$. Observe that it does not depend on the choice of the reduced expression.

Note that the Baumslag-Solitar groups are one of the easiest examples of $\HNN$ extensions. Indeed, we have that $\BS(n,m)=\HNN(\Z,n\Z,\theta)$, where $\theta(n)=m$. 

\begin{lemma}\label{lem.centralizer}
Let $n,m\in\Z$ satisfy $2\leq n \leq |m|$. The centralizer $\cC$ of $\langle a^n\rangle$ inside $\BS(n,m)$ is nonamenable.
\end{lemma}
\begin{proof}
We have the following three cases to consider.

\textbf{Case 1.\ ($|m|\neq 2$).} Define $G := \langle a^\Z,b^{-1}a^\Z b\rangle \leq \BS(n,m)$. We first show that $G$ is an amalgamated free product of two copies of $\Z$ over a copy of $\Z$ embedded as $n\Z$ and $m\Z$ respectively. Write $H:=\{c,d\mid c^n=d^m\}$ and define the homomorphism $\alpha:H\rightarrow G$ by $\alpha(c)=a$ and $\alpha(d)=b^{-1}ab$. Note that $\alpha$ is well defined and surjective. We show that $\alpha$ is also injective. To that end we fix $g\in \Ker(\alpha)$. Write $g=c^{n_0}d^{m_1}c^{n_1}\ldots d^{m_k}c^{n_k}$ with $k\geq 0$, $n_0\in\Z$, $n_1,\ldots,n_k\in \Z\setminus n\Z$ and $m_1,\ldots,m_k\in\Z\setminus m\Z$. Then,
\begin{align*}
e = \alpha(g)
&= \alpha(c^{n_0}d^{m_1}c^{n_1}\ldots d^{m_k}c^{n_k})\\
&= \alpha(c)^{n_0}\alpha(d)^{m_1}\alpha(c)^{n_1}\ldots \alpha(d)^{m_k}\alpha(c)^{n_k}\\
&= a^{n_0}b^{-1}a^{m_1}ba^{n_1}\ldots b^{-1}a^{m_k}ba^{n_k}.
\end{align*}
Note that this last expression is reduced inside $\BS(n,m)$. So, by Lemma \ref{lem.Brit}, we have that $k=0$ and $n_0=0$. But then $g=e$ and hence $\Ker(\alpha)=\{e\}$. Altogether, we see that $G$ is indeed an amalgamated free product of two copies of $\Z$ over a copy of $\Z$ embedded as $n\Z$ and $m\Z$ respectively. In particular $G$ is nonamenable by the remark following Proposition 23 of \cite{dlHP11}. Since $G$ is a subgroup of $\cC$, we have that $\cC$ is also nonamenable.

\textbf{Case 2.\ ($m=2$).} So $n=m=2$. In this case $\cC$ and $\BS(n,m)$ coincide. In particular, $\cC$ is nonamenable.

\textbf{Case 3.\ ($m=-2$).} So $n=2$ and $m=-2$. Let $g\in \BS(n,m)\setminus \cC$. Then $a^2g=ga^{-2}$, and hence $a^2gb=ga^{-2}b=gba^2$. This shows that $\BS(n,m) = \cC \sqcup \cC b^{-1}$. In other words, $\cC$ is an index $2$ normal subgroup of $\BS(n,m)$. In particular, $\cC$ is nonamenable.
\end{proof}

Since $\BS(n,m)$ is an $\HNN$ extension, we have that it acts on its \emph{Bass-Serre tree}. Recall from \cite{Se80} that the Bass-Serre tree $T$ is defined as follows:
\[V(T)=\BS(n,m)/\langle a \rangle \text{ and } E^+(T)=\BS(n,m)/\langle a^n\rangle,\]
where $V(T)$ denotes the set of vertices of $T$ and $E^+(T)$ denotes the set of positive oriented edges of $T$. The \emph{source} map $s:E^+(T)\rightarrow V(T)$ and the \emph{range} map $r:E^+(T)\rightarrow V(T)$ are defined by
\[s(g\langle a^n\rangle)=g\langle a\rangle \text{ and } r(g\langle a^n\rangle)=gb^{-1}\langle a\rangle \text{ for every } g\in\BS(n,m).\]
The group $\BS(m,n)$ then acts on $T$ by left multiplication.

In general, when $\Gamma$ is a group acting on a tree $T$, we call a group element \emph{elliptic} if it admits a fixed point, otherwise we call it \emph{hyperbolic}. The following lemma is well known and follows immediately from \cite[Proposition 25 and Proposition 26]{Se80}.

\begin{lemma}\label{lem.Hyperb}
Let $\Gamma$ be a group acting on a tree $T$. 
\begin{enumerate}
\item If $g\in\Gamma$ is a hyperbolic element, then $g^z$ is hyperbolic for every nonzero integer $z$.
\item If $g,h\in\Gamma$ are elliptic elements such that $gh$ is elliptic, then $g$ and $h$ have a common fixed point.
\end{enumerate}
\end{lemma}

\subsection{Almost normal subgroups and quasi-centralizers}

Let $\Gamma$ be group and $\Lambda < \Gamma$ a subgroup. Define the following functions on $\Gamma$ having values in $\N\cup\{\infty\}$:
\begin{enumerate}
\item $r(g)=[\Lambda:\Lambda\cap g\Lambda g^{-1}]=$ the number of right $\Lambda$-cosets in the double coset $\Lambda g \Lambda$;
\item $l(g)=r(g^{-1})=[\Lambda:\Lambda\cap g^{-1}\Lambda g]=$ the number of left $\Lambda$-cosets in the double coset $\Lambda g \Lambda$.
\end{enumerate}

If $l(g)$ is finite for every $g\in\Gamma$, we say that $\Lambda$ is an \emph{almost normal} subgroup of $\Gamma$. In that case, we also call $(\Gamma,\Lambda)$ a \emph{Hecke pair}.

\begin{remark}
In the literature, e.g. \cite{Tz03}, the function $r$ is usually denoted by $L$ and the function $l$ is usually denoted by $R$. Let us justify our choice of notation. Let $(P,\tau)$ be a tracial von Neumann algebra and let $\BS(n,m)\actson P$ be a trace preserving action. Set $N:=P\rtimes \langle a \rangle$ and define the $N$-$N$-bimodule $\cK_g:=\overline{\lspan}~Nu_gN$ for every $g\in\BS(n,m)$. Then the left dimension $\dim_{N-}(\cK_g)$ equals $l(g)$ and the right dimension $\dim_{-N}(\cK_g)$ equals $r(g)$.
\end{remark}

%
%
%

Note that $(\BS(n,m),\langle a \rangle)$ is a Hecke pair. When considering this pair, $l(g)$ is the smallest nonzero positive integer such that $ga^{l(g)}\in \langle a\rangle g$. Similarly, $r(g)$ is the smallest nonzero positive integer such that $a^{r(g)}g\in g\langle a\rangle$. Writing $k=\gcd(|n|,|m|)$, $n_0=n/k$, $m_0=m/k$ and
\[\cF:=\{k|n_0|^s|m_0|^t\mid s,t\in\N \text{ with } s+t>0\},\]
we have that $\cF\cup\{1\}=\{l(g)\mid g\in\BS(n,m)\}$.

We end this subsection with the notion of quasi-centralizer. The \emph{quasi-centralizer} $\QC_\Gamma(\Lambda)$ of an inclusion of groups $\Lambda\leq \Gamma$ is defined as
\[\QC_\Gamma(\Lambda) := \bigcup_{\substack{\Lambda_1\leq \Lambda\\\text{ finite index}}} \Cent_\Gamma(\Lambda_1),\]
where $\Cent_\Gamma(\Lambda_1)$ denotes the centralizer of $\Lambda_1\leq \Gamma$. Note that $\QC_\Gamma(\Lambda)$ is a normal subgroup of $\Gamma$ whenever $\Lambda$ is an almost normal subgroup of $\Gamma$. 

\begin{lemma}\label{lem.QC}
Let $n$ and $m$ be nonzero integers. Then,
\[\QC_{\BS(n,m)}(\langle a \rangle) = \{g\in \BS(n,m)\mid ga^{l(g)}g^{-1}=a^{l(g)}\}.\]
\end{lemma}
\begin{proof}
It suffices to show that 
$$\QC_{\BS(n,m)}(\langle a \rangle) \subset \{g\in \BS(n,m)\mid ga^{l(g)}g^{-1}=a^{l(g)}\},$$ 
since the converse inclusion is obvious. So fix an element $g$ of the quasi-centralizer of $\langle a\rangle$ inside $\BS(n,m)$. Then there exists a nonzero positive integer $l$ such that $ga^lg^{-1}=a^l$. This implies that $a^l \in \langle a\rangle \cap g^{-1}\langle a \rangle g = \langle a^{l(g)}\rangle$ or in other words, $l(g)$ divides $l$. Writing $l=l_0l(g)$ we have that 
$$a^l=ga^lg^{-1}=(ga^{l(g)}g^{-1})^{l_0}.$$ 
Now $ga^{l(g)}g^{-1}=a^{r}$, for some $r\in\{r(g),-r(g)\}$. So altogether $a^l=a^{rl_0}$ and therefore $r=l/l_0=l(g)$. We conclude that $ga^{l(g)}g^{-1}=a^r=a^{l(g)}$.
\end{proof}

\section{Controlling relative commutants}\label{sec.relcom}


Fix integers $n,m\in \Z$ such that $2\leq n\leq |m|$. Let $(P,\tau)$ be a diffuse tracial von Neumann algebra and let $\BS(n,m)\actson P$ be a trace preserving action such that $P\rtimes \langle a^{l(g)}\rangle \subset P\rtimes \BS(n,m)$ is irreducible for every $g\in\BS(n,m)$. We write $\Gamma:=\BS(n,m)$, $M:=P\rtimes \Gamma$, $N:=P\rtimes \langle a\rangle$ and $N_z:=P\rtimes \langle a^z\rangle$ for every nonzero integer $z$.

The following theorem is our main technical result and is, as we mentioned before, heavily inspired by Lemma 8.4 from \cite{IPP08}. Roughly speaking, it lets us control relative commutants in $M$ of irreducible finite index von Neumann subalgebras of $N$, allowing us later to deduce unitary conjugacy from a two-sided intertwining.

\begin{theorem}\label{thm.Tech}
Let $p$ be a nonzero projection of $N$. Let $Q\subset pNp$ be an irreducible finite index inclusion. Then there exists a unitary $u\in\cU(pMp)$ such that $uQu^*\subset pNp$ and $u(Q'\cap pMp)u^*\subset pNp$.
\end{theorem}
\begin{proof}
Since $N$ is a II$_1$ factor and $P$ is a diffuse von Neumann subalgebra, we may actually assume that the projection $p$ from the description of the theorem is an element of $P$. So let $p$ be a nonzero projection of $P$ and let $Q\subset pNp$ be an irreducible finite index inclusion. 

For every $g\in\Gamma$, we denote by $\cK_g$ the $N$-$N$-subbimodule $\overline{\lspan}\,Nu_gN$ of $L^2(M)$. Note that $\cK_g$ is the closed linear span of $\{bu_h\mid b\in P,\; h\in \langle a\rangle g\langle a\rangle\}$. So $\cK_g$ and $\cK_h$ coincide if and only if $\langle a\rangle g\langle a\rangle=\langle a\rangle h\langle a\rangle$. Define $\hat{\Gamma}=\langle a\rangle \backslash \Gamma /\langle a\rangle$, i.e.\ the set of all double classes of $\langle a \rangle \leq \Gamma$. Then clearly 
$$\bim{N}{L^2(M)}{N}=\bigoplus_{\langle a\rangle g\langle a\rangle\in\Gammahat} \bim{N}{(\cK_g)}{N}.$$
For every $g\in\Gamma$, we denote by $p_g$ the orthogonal projection of $L^2(M)$ onto $\cK_g$. 

Whenever $x\in M$, we see that 
$$p_g(x)=\sum_{i=0}^{l(g)-1} E_N(xu_{ga^{i}}^*)u_{ga^{i}}\in M.$$ 
In particular, if $x\in Q'\cap pMp$, then $p_g(x)$ is a $Q$-central vector of $\cK_g\cap pMp$. From this we get that $Q'\cap pMp$ is $||\cdot||_2$-norm densely spanned by the $Q$-central vectors of $\cK_g\cap pMp$, where $g$ runs over all elements of $\Gamma$. So investigating $Q'\cap pMp$ comes down to investigating the $Q$-central vectors of $\cK_g\cap pMp$ for every $g\in\Gamma$. To that end, we introduce the sets $\Delta = \{g\in\Gamma\mid \cK_g\cap pMp \text{ has a nonzero } Q\text{-central vector}\}$ and $\hat{\Delta}=\{\langle a\rangle g \langle a\rangle\in\Gammahat\mid g\in \Delta\}$. 

Since $Q\subset pNp$ has finite index and $pNp\subset pMp$ is irreducible, we have that $Q'\cap pMp$ is finite dimensional (see e.g.\ \cite[Lemma A.3]{Va08}). On the other hand, different elements of $\hat{\Delta}$ give rise to orthogonal nonzero elements of $Q'\cap pMp$. Altogether we see that $\hat{\Delta}$ is a finite set, say $\hat{\Delta}=\{\langle a \rangle g_1\langle a \rangle,\ldots,\langle a \rangle g_{\kappa}\langle a \rangle\}$. We also have the following claim.

\textbf{Claim 1.\ If $g\in\Delta$, then $l(g)=r(g)$.}

Let $g\in\Delta\setminus \langle a \rangle$ and let $x$ be a nonzero $Q$-central vector of $\cK_g\cap pMp$. For every integer $z>0$, we define

\[\cK_g^{\otimes z}:=\underbrace{\cK_g\otimes_N\ldots\otimes_N\cK_g}_{z\text{ times}}.\]

We first show that $x^{\otimes z}\in (\cK_g)^{\otimes z}$ is a nonzero element. Note that $E_{N}(x^*x)$ is an element of $Q'\cap pNp=\C p$. Hence $E_{N}(x^*x)=(||x||_{L^2(M)}^2/\tau(p))p=||x||_{L^2(pMp)}^2p$. Therefore, for every integer $z>0$, the element $x^{\otimes z}\in (\cK_g)^{\otimes z}$ satisfies $||x^{\otimes z}||=||x||_{L^2(M)}||x||_{L^2(pMp)}^{z-1}$. In particular, $x^{\otimes z}$ is indeed a nonzero element of $(\cK_g)^{\otimes z}$.

Now, for every integer $z>0$, we define the nonzero $pNp$-$pNp$-subbimodule 
\[\cH_z:= \overline{\lspan}\, pN(x^{\otimes z})Np \subset p(\cK_g^{\otimes z})p.\]

By Lemma \ref{lem.Subbim}, we have for every integer $z>0$ that $\cH_z$ is isomorphic with a $pNp$-$pNp$-subbimodule of $L^2(pNp)\otimes_Q L^2(pNp)$. By Proposition \ref{Prop.dimConnes}, $L^2(pNp)\otimes_Q L^2(pNp)$ is a finite index $pNp$-$pNp$-bimodule. So, we find using Lemma \ref{lem.finitesubbim} that $L^2(pNp)\otimes_Q L^2(pNp)$ only has a finite number of nonisomorphic $pNp$-$pNp$-subbimodules. Hence, there exist two nonzero positive integers $z_1$ and $z_2$ such that $z_1\neq z_2$ and $\bim{pNp}{(\cH_{z_1})}{pNp}\cong\bim{pNp}{(\cH_{z_2})}{pNp}$.

Let us conclude the proof of Claim $1$ by showing that for every integer $z>0$ and every nonzero $pNp$-$pNp$-subbimodule $\cH$ of $p(\cK_g^{\otimes z})p$, we have 
$$\dim_{pNp-}(\cH)/\dim_{-pNp}(\cH)=(l(g)/r(g))^z.$$

So let $z>0$ be an integer. We define $l=k(n_0^sm_0^t)^z$, where $s$ and $t$ satisfy $l(g)=kn_0^sm_0^t$. We define $r$ analogously. For $0\leq i_1,\ldots,i_z < l(g)$, we introduce the $pNp$-$pN_lp$-subbimodule $\cH_{(i_1,\ldots,i_z)}$ of $p(\cK_g^{\otimes z})p$ as
$$\bim{pNp}{(\cH_{(i_1,\ldots,i_z)})}{pN_lp} = \bim{pNp}{(\overline{pN(u_{ga^{i_1}}\otimes\ldots\otimes u_{ga^{i_z}})p})}{pN_lp}.$$
Since, for every projection $q\in P$ and every $0\leq i,j<l(g)$,
$$E_N(u_{ga^{i}}qu_{ga^{j}}^*)=\begin{cases} u_{ga^{i}}qu_{ga^{i}}^* & \text{if } i=j\\
0 &\text{otherwise,}\end{cases}$$
we have that the bimodules $\cH_{(i_1,\ldots,i_z)}$ are pairwise orthogonal. Furthermore, since $\cK_g$ is spanned by $L^2(N)u_{g},\ldots, L^2(N)u_{ga^{(l(g)-1)}}$ we have that $p(\cK_g^{\ot z})p$ is spanned by the bimodules $\cH_{(i_1,\ldots,i_z)}$. Altogether, we have found that
$$\bim{pNp}{(p(\cK_g^{\ot z})p)}{pN_lp} = \bigoplus_{(i_1,\ldots,i_z)} \bim{pNp}{(\cH_{(i_1,\ldots,i_z)})}{pN_lp}.$$

Fix $0\leq i_1,\ldots,i_z < l(g)$ and write $\beta$ for $\Ad(u_{ga^{i_1}\ldots ga^{i_z}})$ on $pN_lp$, then

\[\bim{pNp}{(\cH_{(i_1,\ldots,i_z)})}{pN_lp}\cong \bim{pNp}{\cH(\beta)}{pN_lp},\]

where $\cH(\beta)$ is given by $\cH(\beta)=pL^2(N)\beta(p)$ and $x\xi y = x \xi \beta(y)$. Since $\beta(pN_lp)=\beta(p)N_r\beta(p)$ and $\beta(p)N_r\beta(p) \subset \beta(p)N\beta(p)$ is an irreducible inclusion, we have that the bimodule $\cH(\beta)$ is irreducible. Also, its left dimension is $1$ and its right dimension is $r$. We find from all of this that $p(\cK_g^{\otimes z})p$ is orthogonally spanned by irreducible $pNp$-$pN_lp$-subbimodules having left dimension $1$ and right dimension $r$. Since every nonzero $pNp$-$pN_lp$-subbimodule $\cK$ of $p(\cK_g^{\otimes z})p$ is a direct sum of irreducible $pNp$-$pN_lp$-subbimodules of $p(\cK_g^{\otimes z})p$, we have that every nonzero $pNp$-$pN_lp$-subbimodule $\cK$ of $p(\cK_g^{\otimes z})p$ satisfies
\[\dim_{pNp-}(\cK)/\dim_{-pN_lp}(\cK)=1/r.\]
Now let $\cH$ be a nonzero $pNp$-$pNp$-subbimodule of $p(\cK_g^{\otimes z})p$. By Proposition 2.3.5 of \cite{JS97}, we have that $\dim_{-pN_lp}(\cH) = [pNp:pN_lp] \dim_{-pNp}(\cH)$ and therefore
\begin{align*}
\dim_{pNp-}(\cH)/\dim_{-pNp}(\cH)&=[pNp:pN_lp](\dim_{pNp-}(\cH)/\dim_{-pN_lp}(\cH))\\
&=l/r=(l(g)/r(g))^z
\end{align*}

This concludes the proof of Claim $1$.

We continue with the proof of the theorem. Take $\ntil=\lcm(\{l(g)\mid g\in\Delta\})$. Note that $\ntil \in \{l(g)\mid g\in \Gamma\}$. By Claim $1$, we have that $u_gN_{\ntil} u_g^*=N_{\ntil}$ for every $g\in \Delta$. Let us now view $Q$ inside an amplification of $N_{\ntil}$ in the following sense. Since $Q\subset pNp$ and $N_{\ntil}\subset N$ are finite index inclusions, there exists an integer $d>0$, a projection $q\in N_{\ntil}^d$, a normal $*$-homomorphism $\psi:Q\rightarrow q N_{\ntil}^dq$ and a nonzero partial isometry $v\in (M_{1,d}(\C)\otimes pN)q$ such that
\begin{itemize}
\item $\psi(Q)\subset q N_{\ntil}^dq$ is finite index;
\item $v\psi(x)=xv$ for every $x\in Q$.
\end{itemize}
Since $\psi(Q)\subset q N_{\ntil}^dq$ is finite index and $q N_{\ntil}^dq$ is a factor, we have by \cite[Lemma A.3]{Va08} that $\psi(Q)'\cap q N_{\ntil}^dq$ is finite dimensional. Cutting $q N_{\ntil}^dq$ with a minimal projection of $\psi(Q)'\cap q N_{\ntil}^dq$, we may actually assume that $\psi(Q)\subset q N_{\ntil}^dq$ is irreducible.

Since $vv^*\in Q'\cap pNp =\C p$, we have that $vv^*=p$. Let us take a closer look at the inclusion $v^*(Q'\cap pMp)v\subset qM^dq$. Recall that $\hat {\Delta}=\{\langle a \rangle g_1\langle a \rangle,\ldots,\langle a\rangle g_\kappa\langle a\rangle\}$. For $1\leq \alpha\leq \kappa$, let $x_\alpha$ be a nonzero $Q$-central vector of $\cK_{g_\alpha}\cap pMp$. Note that $v^*x_\alpha v$ is a nonzero $\psi(Q)$-central vector of $q (M_d(\C)\otimes \cK_{g_\alpha})q \cap qM^dq$. Furthermore
\[q (M_d(\C)\otimes \cK_{g_\alpha})q=\bigoplus_{i=0}^{\ntil-1}\bigoplus_{j=0}^{l(g_\alpha)-1}q(L^2(N_{\ntil}^d)(1\otimes u_{a^{i}g_\alpha a^j}))q,\]
as $q N_{\ntil}^dq$-$q N_{\ntil}^dq$-bimodules. Since $Q'\cap pMp$ is spanned by the $Q$-central vectors of $\bigcup_{i=1}^{\kappa}(\cK_{g_i}\cap pMp)$, we find from all this that 
\begin{equation}\label{eq.Omega}
v^*(Q'\cap pMp)v\subset \lspan\{q(N_{\ntil}^d(1\otimes u_{h}))q\mid h\in\cL\}
\end{equation}
for some finite set $\cL\subset \Gamma$. Moreover we see that $\cL$ can be chosen to lie in $\Omega$, where
$$\Omega = \{g\in\Gamma\mid u_gN_{\ntil} u_g^*=N_{\ntil} \text{ and } q(N_{\ntil}^d(1\otimes u_g))q \text{ has a nonzero } \psi(Q)\text{-central vector}\}.$$
Note that $\Omega$ is a group. Indeed, by the irreducibility of $\psi(Q)\subset q N_{\ntil}^dq$ we have that $\Omega$ coincides with the group
$$\{g\in\Omega\mid q(N_{\ntil}^d(1\otimes u_g))q\text{ has a } \psi(Q)\text{-central vector } x \text{ with } xx^*=q=x^*x\}.$$
We also have the following claim.

\textbf{Claim 2.\ For every finite subset $\cL\subset \Omega$, there exists an element $g_0\in \Gamma$ such that $r(g_0)\mid\ntil$ and $\cL\subset g_0\langle a\rangle g_0^{-1}$.}

We first show that every element of $\Omega$ is elliptic with respect to the action of $\Gamma$ on its Bass-Serre tree $T$ (see Section \ref{Sec.HNN}). So let $g$ be an element of $\Omega$. By definition of $\Omega$, there exists a $\psi(Q)$-central vector $x$ of $q(N_{\ntil}^d(1\otimes u_g))q$ satisfying $u_gN_{\ntil} u_g^*=N_{\ntil}$ and $xx^*=q=x^*x$. Note that $\Ad(x)\in \Aut(q N_{\ntil}^d q)$ satisfies $\Ad(x)|_{\psi(Q)}=\id|_{\psi(Q)}$. Since $\psi(Q)\subset q N_{\ntil}^d q$ is finite index and irreducible, the group of automorphisms of $q N_{\ntil}^d q$ that restrict to the identity on $\psi(Q)$ is finite (see e.g.\ \cite[Lemma 8.12]{Fa09}). Therefore, there exists an integer $z>0$ such that $\Ad(x)^z=\id$ on $qN_{\ntil}^dq$. Equivalently $x^z\in (q N_{\ntil}^d q)'\cap q M^d q=\C q$. On the other hand, $x^z$ is a nonzero element of $q (N_{\ntil}^d(1\otimes u_{g^z})) q$. Therefore, $g^z$ must be an element of $\langle a^{\ntil}\rangle$ and hence $g^z$ must be an elliptic element. Lemma \ref{lem.Hyperb}.(1) implies that $g$ must also be elliptic.

Now let $\cL=\{g_1,\ldots,g_t\}$ be a finite subset of $\Omega$. Let $g,h\in\Omega$ be arbitrary elements. As we just showed, $g$ and $h$ are elliptic. Since $\Omega$ is a group, we have that $gh\in\Omega$ and hence also $gh$ is elliptic. Using Lemma \ref{lem.Hyperb}.(2), we get that the fixed point sets of $g$ and $h$ intersect nontrivially whenever $g,h\in\Omega$. Write $T_i=\{x\in V(T)\mid g_i\cdot x=x\}$ for $1\leq i \leq t$. Then $T_i$ is a subtree of $T$. Furthermore we already know that $T_i\cap T_j\neq \emptyset$ for every $1\leq i,j\leq t$. It is an easy exercise to verify that finitely many subtrees of a given tree with pairwise nontrivial intersections have a nontrivial intersection. Hence there exists a vertex $x\in V(T)$ such that $g_i\cdot x = x$ for every $i$. In other words, there exists an element $g\in\Gamma$ such that $g_i\in g\langle a\rangle g^{-1}$ for every $i$. Define $\cJ=\{g\in\Gamma\mid \cL\subset g\langle a\rangle g^{-1}\}$. We already showed that $\cJ$ is nonempty. Choose an element $g_0\in\cJ$ having minimal $b$-length (see Section \ref{Sec.HNN}). Using Lemma \ref{lem.Brit} and the fact that $r(g)\mid \ntil$ for every $g\in\Omega$, it follows that $r(g_0)\mid \ntil$. This concludes the proof of Claim $2$.

Let us now finish the proof of the theorem. Combining inclusion (\ref{eq.Omega}) with Claim 2, we find that there exists an element $g_0\in\Gamma$ such that $r(g_0)\mid \ntil$ and $v^*(Q'\cap pMp)v$ is a subset of $\lspan\{q(N_{\ntil}^d(1\otimes u_{h}))q\mid h\in g_0\langle a \rangle g_0^{-1}\}$.

Set $\tilde{v}=v(1\otimes u_{g_0})$. Then
\begin{align*}
\tilde{v}^*(Q'\cap pMp)\tilde{v}&=(1\otimes u_{g_0}^*)v^*(Q'\cap pMp)v(1\otimes u_{g_0})\\
&\subset \lspan\{(1\otimes u_{g_0}^*)(N_{\ntil}^d(1\otimes u_{h}))(1\otimes u_{g_0})\mid h\in g_0\langle a \rangle g_0^{-1}\}\\
&\subset \lspan\{(N^d(1\otimes u_{g_0^{-1}hg_0})\mid h\in g_0\langle a \rangle g_0^{-1}\}\\
&= N^d.
\end{align*}
In particular $\tilde{v}^*\tilde{v}\in N^d$. Furthermore we have that
\begin{align*}
\tilde{v}^*Q\tilde{v}&=\tilde{v}^*\tilde{v}(1\otimes u_{g_0}^*)\psi(Q)(1\otimes u_{g_0})\\
&\subset N^d(1\otimes u_{g_0}^*)N_{\ntil}^d(1\otimes u_{g_0})\\
&= N^d.
\end{align*}
Let $w$ be an element of $M_{d,1}(\C)\otimes N$ such that $ww^*=\tilde{v}^*\tilde{v}$ and $w^*w=\tilde{v}\tilde{v}^*=p$. Define $u=(\tilde{v}w)^*\in pMp$. Then $u\in \cU(pMp)$ and also
\begin{align*}
uQu^* 
&= w^* \tilde{v}^* Q \tilde{v} w\\
&\subset w^* N^d w\\
&\subset N
\end{align*}
and
\begin{align*}
u (Q'\cap pMp) u^*
&= w^* \tilde{v}^* (Q'\cap pMp) \tilde{v} w\\
&\subset w^* N^d w\\
&\subset N.
\end{align*}
This ends the proof.
\end{proof}

In the proof of Theorem \ref{thm.Tech}, we used the following two well known results.

\begin{lemma}\label{lem.Subbim}
Let $(M,\tau)$ be a II$_1$ factor and let $N\subset M$ be an irreducible finite index subfactor. Let $\bim{M}{\cH}{M}$ be an $M$-$M$-bimodule and let $\xi\in\cH$ be a nonzero $N$-central vector such that $\lspan M\xi M$ is dense in $\cH$. Then $\cH$ is isomorphic with an $M$-$M$-subbimodule of $L^2(M)\otimes_N L^2(M)$.
\end{lemma}
\begin{proof}
By replacing $\xi$ with $\xi/||\xi||$, we may assume that $\xi\in\cH$ is a unit vector. Let us begin by showing that $L^2(M)$ and $\overline{\xi M}$ are isomorphic as $N$-$M$-bimodules. For that define $\varphi:M\rightarrow \C:x\mapsto \langle \xi x,\xi\rangle$. Then $\varphi$ is a normal $N$-central state on $M$. Since $N\subset M$ is irreducible, this implies that $\varphi$ and $\tau$ coincide. Hence
$$||\xi x||^2=\langle \xi x , \xi x\rangle = \varphi(xx^*)=\tau(xx^*)=||x||_2^2.$$
From this we find that we can extend the map $\alpha:M\rightarrow \overline{\xi M}:x\mapsto \xi x$ to a unitary from $L^2(M)$ onto $\overline{\xi M}$. This unitary is $N$-$M$-bimodular by construction.

Let us continue with the proof of the lemma. Use Lemma 2.8 from \cite{FR12} where $\bim{M}{\cK}{N}:=\bim{M}{L^2(M)}{N}$ and $\bim{N}{\cL}{M}=\bim{N}{\overline{\xi M}}{M}\subset \bim{N}{\cH}{M}$. Then we get that $\cH$ is an $M$-$M$-bimodule isomorphic with a subbimodule of $L^2(M)\otimes_N \overline{\xi M}$. Since we already showed that $L^2(M)$ and $\overline{\xi M}$ are isomorphic as $N$-$M$-bimodules, we are done.
\end{proof}

\begin{lemma}\label{lem.finitesubbim}
Let $(M,\tau)$ be a II$_1$ factor and let $\bim{M}{\cH}{M}$ be a nonzero finite index bimodule. Then $\cH$ only contains a finite number of nonisomorphic $M$-$M$-subbimodules.
\end{lemma}
\begin{proof}
Let $(M,\tau)$ be a II$_1$ factor and let $\bim{M}{\cH}{M}$ be a finite index bimodule. Then there exists a nonzero positive integer $n$, a nonzero projection $p\in M^n$ and a normal $*$-homomorphism $\psi:M\rightarrow pM^np$ such that $[pM^np:\psi(M)]<\infty$ and
\[\bim{M}{\cH}{M} \cong \bim{M}{\cH(\psi)}{M},\]
where $\bim{M}{\cH(\psi)}{M}$ is given by $\cH(\psi)= p(\C^n\otimes L^2(M))$ and $x\xi y = \psi(x)\xi y$. In this way, we see that the isomorphism classes of the $M$-$M$-subbimodules of $\cH$ correspond with the equivalence classes of the projections in $\psi(M)'\cap pM^np$. On the other hand, $\psi(M)'\cap pM^np$ is finite dimensional by \cite[Lemma A.3]{Va08}. Hence 
$$\{(\Tr_{M_n(\C)}\otimes\tau)(q)\mid q\in \psi(M)'\cap pM^np \text{ is a projection}\}$$ 
is a finite set. To end the proof, Proposition 1.1.2 of \cite{JS97} states that $\Tr_{M_n(\C)}\otimes\tau$ is a complete invariant for the equivalence classes of projections in $M^n$, since $M^n$ is a factor.
\end{proof}

\section{Unitary conjugacy of the canonical subalgebras $\alpha(pN_1p)$ and $N_2$}\label{sec.Uni}

Fix integers $n_1,m_1,n_2,m_2\in\Z$ satisfying $2\leq n_1\leq |m_1|$ and $2\leq n_2\leq |m_2|$. For $i\in\{1,2\}$, let $(P_i,\tau_i)$ be a diffuse amenable tracial von Neumann algebra and let $\BS(n_i,m_i)\actson P_i$ be a trace preserving action such that $P_i\rtimes \langle a_i^{l(g)}\rangle \subset P_i\rtimes \BS(n_i,m_i)$ is irreducible for every $g\in \BS(n_i,m_i)$. We still denote by $\tau_i$ the canonical trace of $P_i\rtimes \BS(n_i,m_i)$.

We make use of the following notation: $\Gamma_i=\BS(n_i,m_i)$, $M_i=P_i\rtimes \Gamma_i$, $N_i=P_i\rtimes \langle a_i\rangle$ and $N_{i,z}=P_i\rtimes \langle a_i^z\rangle$ for every nonzero integer $z$.

\begin{proposition}\label{Bimod}
Let $p$ be a nonzero projection of $N_1$. If $\alpha:pM_1p\rightarrow M_2$ is an isomorphism, then there exists a nonzero irreducible finite index $\alpha(pN_1p)$-$N_2$-subbimodule of $L^2(M_2)$.
\end{proposition}
\begin{proof} Since $P_1$ is diffuse and $N_1$ is a II$_1$ factor, we may actually assume that $p$ is a nonzero projection of $P_1$. So let $p$ be a nonzero projection of $P_1$ and let $\alpha:pM_1p\rightarrow M_2$ be an isomorphism.

We first prove that $N_2\prec \alpha(pN_1p)$. Denote by $\cC$ the centralizer of $\langle a_2^{n_2}\rangle$ inside $\Gamma_2$. Recall that by Lemma \ref{lem.centralizer} $\cC$ is nonamenable. Hence the group von Neumann algebra $L(\cC)$ has no amenable direct summand. By Proposition 3.1 in \cite{Ue07}, the HNN extension $M_1=\HNN(N_1,N_{1,n_1},\Ad(u_{b_1}))$ can be viewed as the corner of an amalgamated free product of tracial von Neumann algebras. Since $L(\cC)$ has no amenable direct summand, it then follows from \cite[Theorem 4.2]{CH08} that
$$\alpha^{-1}(L(\cC))'\cap pM_1p\prec N_1.$$ 
Since $\alpha^{-1}(L(\langle a_2^{n_2}\rangle))$ is a subalgebra of $\alpha^{-1}(L(\cC))'\cap pM_1p$ this implies that $\alpha^{-1}(L(\langle a_2^{n_2}\rangle))\prec N_1$. Combining that with Lemma \ref{lem.interfi} and the fact that $\alpha^{-1}(L(\langle a_2^{n_2}\rangle))\subset \alpha^{-1}(L(\langle a_2\rangle))$ is a finite index inclusion, we get that $\alpha^{-1}(L(\langle a_2\rangle))\prec N_1$. On the other hand, using Theorem 4.1 from \cite{Va13}, we find that $\alpha^{-1}(P_2)\prec N_1$. In fact, by Lemma \ref{Lem.full}, we even have that $\alpha^{-1}(P_2)\prec^f N_1$, since $P_2\subset M_2$ is a regular inclusion. Applying Lemma \ref{lem.Mihai} to $\alpha^{-1}(L(\langle a_2\rangle))\prec N_1$ and $\alpha^{-1}(P_2)\prec^f N_1$, we obtain that $\alpha^{-1}(N_2)\prec_{M_1} N_1$. Since $N_1$ is a factor, this implies that $\alpha^{-1}(N_2)\prec_{M_1} pN_1p$ or equivalently $N_2\prec_{M_2} \alpha(pN_1p)$.

We also show that $\alpha(pN_1p)\prec_{M_2} N_2$. Fix $n\in \N$ such that $n\geq 1/\tau_1(p)$ and choose a projection $q\in P_1$ such that $q\leq p$ and $\tau_1(q)=1/n$. Define the isomorphism $\beta:(\alpha(q)M_2\alpha(q))^n\rightarrow (qM_1q)^n$ given by $1\otimes \alpha^{-1}$. Let $v\in M_{1,n}(\C)\otimes P_1$ satisfy $vv^*=1$ and $v^*v=1\otimes q$. Then $\Ad(v):(qM_1q)^n\rightarrow M_1$ is an isomorphism. Now define the isomorphism $\gamma:(\alpha(q)M_2\alpha(q))^n\rightarrow M_1$ given by $\Ad(v)\circ \beta$. Using the exact same arguments as before, we get that $\gamma^{-1}(N_1)\prec_{M_2^n} N_2^n$ or equivalently that $(\alpha(qN_1q))^n \prec_{M_2^n}N_2^n$. Hence also $\alpha(pN_1p)\prec_{M_2} N_2$.

So far, we have found that $N_2\prec_{M_2} \alpha(pN_1p)$ and $\alpha(pN_1p)\prec_{M_2} N_2$. Since $pN_1p\subset pM_1p$ and $N_2\subset M_2$ are quasi-regular inclusions, a combination of Proposition \ref{prop.full} and Lemma \ref{Lem.full} yields the desired result.
\end{proof}

In the proof of Proposition \ref{Bimod} we used the following lemma which is a slight adaptation of Lemma 2.3 from \cite{BV12}. 

\begin{lemma}\label{lem.Mihai}
Let $\Gamma$ be a countable group and $\Gamma\actson (P,\tau)$ a trace preserving action of $\Gamma$ on a tracial von Neumann algebra $(P,\tau)$. Put $M=P\rtimes \Gamma$ and let $p\in M$ be a projection. Assume that $Q\subset pMp$ is a von Neumann subalgebra that is normalized by a group of unitaries $\cG\subset \cU(pMp)$. Let $\Lambda\leq \Gamma$ be an almost normal subgroup. If $Q\prec_M^f P\rtimes \Lambda$ and $\cG\dpr\prec_M P\rtimes \Lambda$, then $(Q\cup \cG)\dpr \prec_M P\rtimes \Lambda$.
\end{lemma}
\begin{proof}
For every subset $\cF\subset \Gamma$, we denote by $P_{\cF}$ the orthogonal projection of $L^2(M)$ onto the closed linear span of $\{au_g\mid a\in P, g\in\cF\}$. We say that a subset $\cF\subset \Gamma$ is small relative to $\Lambda$ if $\cF$ is contained in a finite union of subsets of the form $g\Lambda h$ with $g,h\in \Gamma$.

Assume, by way of reaching a contradiction, that $(Q \cup \cG)\dpr \nprec P\rtimes \Lambda$. Since $\cU(Q)\cG$ is a group of unitaries generating $(Q\cup \cG)\dpr$, we get from \cite[Lemma 2.4]{Va10} sequences of unitaries $b_n\in U(Q)$ and $w_n\in \cG$ such that $||P_{\cF}(b_nw_n)||_2\rightarrow 0$ for every subset $\cF\subset \Gamma$ that is small relative to $\Lambda$.

Since $\cG\dpr \prec P\rtimes\Lambda$, there exists a nonzero projection $q\in (P\rtimes \Lambda)^n$, a nonzero partial isometry $v\in M_{1,n}(\C)\otimes pM$ and a normal $*$-homomorphism $\psi:\cG\dpr\rightarrow q(P\rtimes \Lambda)^nq$ such that $xv=v\psi(x)$ for all $x\in\cG\dpr$. Denote $p_1=vv^*$ and fix $0< \varepsilon < ||p_1||_2/3$. By the Kaplansky density theorem, we can take a finite subset $\cF_1\subset \Gamma$ and an element $v_1$ in the linear span of $\{au_g\mid a\in M_{1,n}(\C)\otimes P,g\in\cF_1\}$ such that $||v_1||\leq 1$ and $||v-v_1||_2<\varepsilon$.

Denote $\cF_2=\cF_1\Lambda\cF_1^{-1}$. Observe that $\cF_2$ is small relative to $\Lambda$. Write $x_n=v_1\psi(w_n)v_1^*$. By construction, every $x_n$ lies in the image of $P_{\cF_2}$. We also have for all $n$ that $||x_n||\leq 1$ and
\begin{align*}
||w_np_1 -x_n||_2 
&= ||v\psi(w_n)v^* - v_1\psi(w_n)v_1^*||_2\\
&\leq ||v\psi(w_n)v^*-v\psi(w_n)v_1^*||_2 + ||v\psi(w_n)v_1^*-v_1\psi(w_n)v_1^*||_2\\
&\leq ||v\psi(w_n)||\ ||v^*-v_1^*||_2 + ||\psi(w_n)v_1^*||\ ||v-v_1||_2\\
&= ||v||\ ||v^*-v_1^*||_2 + ||v_1||\ ||v-v_1||_2\\ 
&< 2\varepsilon
\end{align*}
Since $Q \prec^f P\rtimes \Lambda$, we obtain from \cite[Lemma 2.5]{Va10} a subset $\cF_3\subset \Gamma$ that is small relative to $\Lambda$ such that $||b_n - P_{\cF_3}(b_n)||_2 < \varepsilon$ for all $n$. In combination with the previous paragraph, we get that
\begin{align*}
||b_nw_np_1 - P_{\cF_3}(b_n)x_n||_2
&\leq ||b_nw_np_1 - b_nx_n||_2 + ||b_nx_n - P_{\cF_3}(b_n)x_n||_2\\
&\leq ||w_np_1 - x_n||_2 + ||x_n||\ ||b_n - P_{\cF_3}(b_n)||_2\\
&< 3\varepsilon,
\end{align*}
for all $n$. Denote $\cF_4=\cF_3\cF_2$. Since $\Lambda$ is an almost normal subgroup of $\Gamma$, we have that $\cF_4$ is still small relative to $\Lambda$. By construction, $P_{\cF_3}(b_n)x_n$ lies in the image of $P_{\cF_4}$ and thus we have shown that $||b_nw_np_1 - P_{\cF_4}(b_nw_np_1)||_2 < 3\varepsilon$ for all $n$.

Since $||P_{\cF}(b_nw_n)||_2\rightarrow 0$ for every subset $\cF\subset \Gamma$ that is small relative to $\Lambda$, it follows from \cite[Lemma 2.3]{Va10} that $||P_{\cF_4}(b_nw_np_1)||_2\rightarrow 0$. Hence $\limsup_n ||b_nw_np_1||_2 \leq 3\varepsilon$. Since $b_n$ and $w_n$ are unitaries, we arrive at the contradiction that $||p_1||_2 \leq 3\varepsilon < ||p_1||_2$.
\end{proof}

The next theorem states that the intertwining bimodule from Proposition \ref{Bimod} can actually be chosen to realize a unitary conjugacy.

\begin{theorem}\label{thm.Conj}
Let $p$ be a nonzero projection of $N_1$. If $\alpha:pM_1p\rightarrow M_2$ is an isomorphism, then $\alpha(pN_1p)$ and $N_2$ are unitarily conjugate in $M_2$.
\end{theorem}
\begin{proof} Let $p$ be a nonzero projection of $N_1$ and let $\alpha:pM_1p\rightarrow M_2$ be an isomorphism. Then by Proposition \ref{Bimod}, there exists an integer $d>0$, a projection $q\in N_2^d$, a normal $*$-homomorphism $\psi:\alpha(pN_1p)\rightarrow qN_2^dq$ and a nonzero partial isometry $v\in (M_{1,d}(\C)\otimes M_2)q$ such that
\begin{itemize}
\item $\psi(\alpha(pN_1p))\subset qN_2^dq$ is irreducible and finite index;
\item $v\psi(x)=xv$ for every $x\in \alpha(pN_1p)$.
\end{itemize}
Identifying $N_2^d$ with $P_2^d\rtimes \langle a_2\rangle$ and identifying $M_2^d$ with $P_2^d\rtimes \Gamma_2$, we can apply Theorem \ref{thm.Tech} to $\psi(\alpha(pN_1p))\subset qN_2^dq \subset qM_2^dq$. 
This yields a unitary $u\in \cU(qM_2^dq)$ such that 
$$u\psi(\alpha(pN_1p))u^*\subset qN_2^dq \text{ and } u(\psi(\alpha(pN_1p))'\cap qM_2^dq)u^*\subset qN_2^dq.$$ Now define $\tilde{\psi}=\Ad(u)\circ \psi$ and $\tilde{v}=vu^*$. Then
\begin{itemize}
\item $\tilde{\psi}:\alpha(pN_1p)\rightarrow qN_2^dq$ is a normal $*$-homomorphism;
\item $\tilde{v}\tilde{\psi}(x)=x\tilde{v}$ for every $x\in \alpha(pN_1p)$.
\end{itemize}
Note that $\tilde{v}\tilde{v}^*$ is a nonzero projection of $\alpha(pN_1p)'\cap M_2=\C 1$. Hence $\tilde{v}\tilde{v}^*$ must be equal to $1$. On the other hand, $\tilde{v}^*\tilde{v}=uv^*vu^*\in u(\psi(\alpha(pN_1p))'\cap qM_2^dq)u^*\subset qN_2^dq$. Let $w$ be an element of $M_{d,1}(\C)\otimes N_2$ such that $ww^*=\tilde{v}^*\tilde{v}$ and $w^*w=\tilde{v}\tilde{v}^*=1$. Then $u_1\in \cU(M_2)$. Also
\begin{align*}
u_1\alpha(pN_1p)u_1^*
&= w^*\tilde{v}^*\alpha(pN_1p)\tilde{v}w = w^*\tilde{v}^*\tilde{v} \tilde{\psi}(\alpha(pN_1p))w\\
&= w^*\tilde{\psi}(\alpha(pN_1p))w \subset w^*N_2^d w\\ 
&\subset N_2.
\end{align*}

By exactly the same arguments, there also exists a unitary $u_2$ inside $pM_1p$ such that $u_2\alpha^{-1}(N_2)u_2^* \subset pN_1p$. Applying $\alpha$ to both sides, there exists a unitary $u_3$ of $M_2$ satisfying $u_3N_2u_3^*\subset \alpha(pN_1p)$. 

Combining both inclusions, we get that $u_1u_3N_2u_3^*u_1^*\subset u_1\alpha(pN_1p)u_1^*\subset N_2$. To finish the proof, it suffices to show that $u_1u_3\in N_2$. To that end, denote the unitary $u_1u_3\in M_2$ by $u_4$ and write $\beta$ for $\Ad(u_4)$ on $N_2$. As before, define $\Gammahat_2=\langle a\rangle \backslash \Gamma_2 /\langle a\rangle$ and $\bim{N_2}{(\cK_g)}{N_2}=\bim{N_2}{(\overline{\lspan}^{||\cdot||_2}N_2u_gN_2)}{N_2}$. Recall that $\bim{N_2}{L^2(M_2)}{N_2}=\bigoplus_{\langle a \rangle g\langle a \rangle}\bim{N_2}{(\cK_g)}{N_2}$. For every $g\in\Gamma_2$, we define $p_{g}$ to be the orthogonal projection of $L^2(M_2)$ onto $\cK_g$. Then we can decompose $u_4$ as
\[u_4=\sum_{\langle a \rangle g \langle a \rangle\in\Gammahat_2} p_g(u_4),\]
where the convergence is in $||\cdot||_2$-norm. Note that $p_{g}(u_4)=\sum_{i=0}^{l(g)-1}x_{g,i}u_{ga^{i}}$, where $x_{g,i}=E_{N_2}(u_4u_{ga^{i}}^*)$. Since $\beta(x)u_4=u_4x$ for every $x\in N_2$, we get that 
\[\beta(x)(\sum_{i=0}^{l(g)-1}x_{g,i}u_{ga^{i}})=(\sum_{i=0}^{l(g)-1}x_{g,i}u_{ga^{i}})x \text{ for every }g\in\Gamma_2\text{ and every }x\in N_2.\] 
But then, for every $g\in\Gamma_2$ and every $x\in N_{2,l(g)}$, we have that $\beta(x)(x_{g,i}u_{ga^{i}})=(x_{g,i}u_{ga^{i}})x$. Therefore $(x_{g,i}u_{ga^{i}})^*(x_{g,i}u_{ga^{i}}) \in N_{2,l(g)}' \cap M_2=\C 1$. So $x_{g,i}$ is a multiple of a unitary for every $g\in\Gamma_2$ and every $0 \leq i < l(g)$.

Now assume that $x_{g,i}$ and $x_{h,j}$ are both nonzero. Note that $(x_{h,j}u_{ha^{j}})^*(x_{g,i}u_{ga^{i}})$ is an element of $N_{2,l}'\cap M_2$, where $l=\lcm(l(g),l(h))$. Since $N_{2,l}$ is irreducible in $M_2$ and $x_{h,j}^*x_{g,i}$ is nonzero, we get that $u_{ha^{j}}\in N_2 u_{ga^{i}}$ and hence that $ha^{j}\in \langle a\rangle ga^{i}$. Therefore, in the decomposition of $u_4$, there is only one nonzero component $x_{g,i}u_{ga^{i}}$. Hence $u_4=xu_g$ for some $g\in\Gamma_2$ and $x\in\cU(N_2)$. But since $u_4N_2u_4^*\subset N_2$, we must have that $g\in\langle a\rangle$ and hence that $u_4\in N_2$. This ends the proof.
\end{proof}

\section{Proof of the main theorem}\label{sec.Main2}

We begin with the following result.

\begin{lemma}\label{lem.Irred}
Let $n,m\in\Z$ such that $2\leq n\leq |m|$. Let $(P,\tau)$ be a tracial von Neumann algebra and let $\BS(n,m)\actson P$ be a trace preserving action such that $P\rtimes \langle a^{l(g)}\rangle\subset P\rtimes \BS(n,m)$ is irreducible for every $g\in \BS(n,m)$. Write $M:=P\rtimes \BS(n,m)$, $N:=P\rtimes \langle a\rangle$ and $\bim{N}{(\cK_g)}{N}:=\bim{N}{(\overline{\lspan}^{||\cdot||_2} Nu_gN)}{N}$. Then 
\begin{itemize}
\item $\bim{N}{(\cK_g)}{N}$ is irreducible for every $g\in \BS(n,m)$;
\item  $\bim{N}{(\cK_g)}{N} \cong \bim{N}{(\cK_h)}{N}$ if and only if $\langle a\rangle g \langle a\rangle = \langle a\rangle h \langle a\rangle$.
\end{itemize}
\end{lemma}
\begin{proof}
As before, we define $N_z:=P\rtimes \langle a^{z}\rangle$ for every nonzero integer $z$. For every $g\in \BS(n,m)$ we have that 
\begin{equation}\label{eq.decomprg}
\bim{N_{r(g)}}{(\cK_g)}{N} = \bigoplus_{i=0}^{r(g)-1} \bim{N_{r(g)}}{(u_{a^{i}g}L^2(N))}{N}\; .
\end{equation}
Let $g,h\in \BS(n,m)$ such that $r(g)=r(h)$. Under the identification (\ref{eq.decomprg}) we have that the set of $N_{r(g)}$-$N$-bimodular elements of $B(\cK_g,\cK_h)$ coincides with
$$B_{g,h}:=\{[x_{i,j}]_{0\leq i,j< r(g)}\mid x_{i,j}\in N_{r(g)}'\cap u_{a^{i}h}Nu_{a^jg}^*\}.$$
Note that $N_{r(g)}'\cap u_{a^{i}h}Nu_{a^jg}^*$ is a subset of $N_{r(g)}'\cap M=\C 1$. Therefore we have that 
$$N_{r(g)}'\cap u_{a^{i}h}Nu_{a^jg}^*=\C 1 \cap u_{a^{i}h}Nu_{a^jg}^* \text{ for every } 0\leq i,j < r(g).$$ 
Hence, 
$$B_{g,h}=\{[x_{i,j}]_{0\leq i,j< r(g)}\mid x_{i,j}\in \C 1 \cap u_{a^{i}h}Nu_{a^jg}^*\}.$$ 
From this, it follows that 
\begin{itemize}
\item $B_{g,g}=\{[x_{i,j}]_{0\leq i,j< r(g)}\mid x_{i,j}\in \C \delta_{i,j}\}$, for every $g\in \BS(n,m)$;
\item $B_{g,h}=\{0\}$, whenever $\langle a\rangle g \langle a\rangle \neq \langle a\rangle h \langle a\rangle$.
\end{itemize}

Let us now prove the first statement. Assume that $\cH$ is a nonzero $N$-$N$-subbimodule of $\cK_g$. Denote by $p_{\cH}$ the orthogonal projection of $\cK_g$ onto $\cH$. Then $p_\cH$ is a nonzero element of $B_{g,g}=\{[x_{i,j}]_{0\leq i,j< r(g)}\mid x_{i,j}\in \C \delta_{i,j}\}$. This implies that $\cH$ contains $u_{a^{i}g}L^2(N)$ for some $0\leq i < r(g)$. Since $\cH$ is also a left $N$-module, we see that $\cH$ coincides with the whole of $\cK_g$. This proves the first statement.

To prove the second statement, assume that $\bim{N}{(\cK_g)}{N} \cong \bim{N}{(\cK_h)}{N}$. Then, by comparing the right dimensions of both bimodules, we have that $r(g)=r(h)$. Furthermore, the unitary between $\cK_g$ and $\cK_h$ is an element of $B_{g,h}$. This implies that $B_{g,h}\neq \{0\}$, and hence that $\langle a\rangle g \langle a\rangle = \langle a\rangle h \langle a\rangle$. This ends also the proof of the second statement.
\end{proof}

For the following lemma, we need to introduce some extra notation. Let $\omega\in\C$ satisfy $|\omega|=1$. Let $(P,\tau)$ be a tracial von Neumann algebra and let $\Z\actson P$ be a trace preserving action. Write $N := P\rtimes \Z$.  Then we define the $*$-automorphism $\alpha_\omega:N\rightarrow N$ by
$$\alpha_\omega(bu_z)=\omega^zbu_z,$$
for every $b\in P$ and $z\in \Z$. Furthermore, we define $\bim{N}{(\cK_\omega)}{N}$ by $\cK_\omega = L^2(N)$ and $x \xi y = \alpha_{\omega}(x)\xi y$.

\begin{lemma}\label{lem.roots}
Let $n,m\in\Z$ such that $2\leq n\leq |m|$. Let $(P,\tau)$ be a tracial von Neumann algebra and let $\BS(n,m)\actson P$ be a trace preserving action such that $P\rtimes \langle a^{l(g)}\rangle\subset P\rtimes \BS(n,m)$ is irreducible for every $g\in \BS(n,m)$. Write $M:=P\rtimes \BS(n,m)$, $N:=P\rtimes \langle a\rangle$, $\omega_g:= e^{2\pi i / r(g)}$ and $\bim{N}{(\cK_g)}{N}:=\bim{N}{(\overline{\lspan}^{||\cdot||_2} Nu_gN)}{N}$. Then
\begin{align}\label{eq.gg}
\bim{N}{(\cK_g\otimes_N\cK_{g^{-1}})}{N}\cong\left(\bigoplus_{i=0}^{r(g)-1} \bim{N}{(\cK_{\omega_g^{i}})}{N}\right) \oplus \left(\bigoplus_{i=1}^{l(g)-1}\bim{N}{(\cK_{ga^{i}g^{-1}})}{N}\right).
\end{align}
Moreover, the bimodules in the decomposition are irreducible and pairwise nonisomorphic.
\end{lemma}
\begin{proof}
To prove (\ref{eq.gg}) we first show that
$$\bim{N}{(\cK_g\otimes_N\cK_{g^{-1}})}{N}=\bigoplus_{i=0}^{l(g)-1} \bim{N}{(\overline{\lspan}(Nu_{ga^{i}}\otimes u_{g^{-1}}N))}{N}.$$
After that we show that $\bim{N}{(\overline{\lspan}(Nu_{ga^{i}}\otimes u_{g^{-1}}N))}{N}$ is isomorphic with $\bim{N}{(\cK_{ga^{i}g^{-1}})}{N}$ when $i\neq 0$ and that $\bim{N}{(\overline{\lspan}(Nu_{g}\otimes u_{g^{-1}}N))}{N}$ is isomorphic with $\bigoplus_{i=0}^{r(g)-1} \bim{N}{(\cK_{\omega_g^{i}})}{N}$.

Let us begin. First of all, we have that $\cK_g\otimes_N\cK_{g^{-1}}$ is linearly spanned by the $N$-$N$-subbimodules $\overline{\lspan}(Nu_{ga^{i}}\otimes u_{g^{-1}}N)$, where $0\leq i < l(g)$. Furthermore, we have for $x,y,z,w\in N$ and $0\leq i,j<l(g)$ that 
\begin{align}\label{eq.innerproduct}
\langle xu_{ga^{i}}\otimes u_{g^{-1}}y , zu_{ga^j}\otimes u_{g^{-1}}w\rangle
&= \langle xu_{ga^{i}}E_N(u_{g^{-1}}yw^*u_g) , zu_{ga^j}\rangle\nonumber\\ 
&= \langle xu_{ga^{i}}u_{g^{-1}}E_{N_{r(g)}}(yw^*)u_g , zu_{ga^j}\rangle\nonumber\\ 
&= \langle xu_{ga^{i}g^{-1}}E_{N_{r(g)}}(yw^*) , zu_{ga^jg^{-1}}\rangle.
\end{align}
If $i\neq j$, then (\ref{eq.innerproduct}) implies that $\overline{\lspan}(Nu_{ga^{i}}\otimes u_{g^{-1}}N)$ and $\overline{\lspan}(Nu_{ga^{j}}\otimes u_{g^{-1}}N)$ are orthogonal. Hence,
$$\bim{N}{(\cK_g\otimes_N\cK_{g^{-1}})}{N}=\bigoplus_{i=0}^{l(g)-1} \bim{N}{\overline{\lspan}(Nu_{ga^{i}}\otimes u_{g^{-1}}N)}{N}.$$
If $i=j\neq 0$, then we can continue with (\ref{eq.innerproduct}) in the following way:
\begin{align*}
\langle xu_{ga^{i}g^{-1}}E_{N_{r(g)}}(yw^*) , zu_{ga^ig^{-1}}\rangle
&= \langle xE_{N}(u_{ga^{i}g^{-1}}yw^*u_{ga^{i}g^{-1}}^*) , z\rangle\\
&= \langle xu_{ga^{i}g^{-1}}yw^*u_{ga^{i}g^{-1}}^* , z\rangle\\
&= \langle xu_{ga^{i}g^{-1}}y,zu_{ga^ig^{-1}}w\rangle.
\end{align*}
This shows that $\bim{N}{\overline{\lspan}(Nu_{ga^{i}}\otimes u_{g^{-1}}N)}{N} \cong \bim{N}{(\cK_{ga^{i}g^{-1}})}{N}$ when $i\neq 0$. Hence,
$$\bim{N}{(\cK_g\otimes_N\cK_{g^{-1}})}{N}\cong  \bim{N}{\overline{\lspan}(Nu_{g}\otimes u_{g^{-1}}N)}{N} \oplus \left(\bigoplus_{i=1}^{l(g)-1}\bim{N}{(\cK_{ga^{i}g^{-1}})}{N}\right).$$
Let us now show that $\overline{\lspan}(Nu_{g}\otimes u_{g^{-1}}N)$ contains an $N$-$N$-subbimodule that is isomorphic with  $\bigoplus_{i=0}^{r(g)-1} \bim{N}{(\cK_{\omega_g^{i}})}{N}$. For that, define $\xi_{g,i}\in\overline{\lspan}(Nu_{g}\otimes u_{g^{-1}}N)$ as
$$\xi_{g,i}:=\frac{1}{\sqrt{r(g)}}\sum_{j=0}^{r(g)-1}(\omega_g^{-i})^j(u_{a^jg}\otimes u_{a^jg}^*).$$
Note that $x\xi_{g,i}=\xi_{g,i}\alpha_{\omega_{g}^{i}}(x)$ for every $x\in N$. Let $x,y\in N$ and $0\leq i,j< r(g)$. We can make the following calculation:
\begin{align}\label{eq.innerproduct2}
\langle \xi_{g,i}x,\xi_{g,j}y\rangle
&= \sum_{k,l=0}^{r(g)-1} \frac{1}{r(g)} \omega_g^{-ki+lj} \langle u_{a^kg}\otimes u_{a^kg}^*x, u_{a^lg} \otimes u_{a^lg}^*y \rangle\nonumber\\
&= \sum_{k,l=0}^{r(g)-1} \frac{1}{r(g)} \omega_g^{-ki+lj} \langle u_{a^kg}E_N(u_{a^kg}^*xy^*u_{a^lg}), u_{a^lg} \rangle\nonumber\\
&= \sum_{k=0}^{r(g)-1} \frac{1}{r(g)} \omega_g^{-ki+kj} \langle u_{a^kg}E_N(u_{a^kg}^*xy^*u_{a^kg}), u_{a^kg} \rangle\nonumber\\
&= \sum_{k=0}^{r(g)-1} \frac{1}{r(g)} \omega_g^{k(j-i)} \tau(xy^*).
\end{align}
If $i \neq j$, then we get from (\ref{eq.innerproduct2}) that $\overline{\xi_{g,i}N}$ is orthogonal to $\overline{\xi_{g,j}N}$. If $i = j$, then (\ref{eq.innerproduct2}) implies that
$$\bim{N}{(\cK_{\omega_g^{i}})}{N} \cong \bim{N}{(\overline{\xi_{g,i}N})}{N}.$$
Hence, we indeed have that $\overline{\lspan}(Nu_{g}\otimes u_{g^{-1}}N)$ contains an $N$-$N$-subbimodule isomorphic with $\bigoplus_{i=0}^{r(g)-1} \bim{N}{(\cK_{\omega_g^{i}})}{N}$. Putting everything together, we have found that $\bim{N}{(\cK_g\otimes_N\cK_{g^{-1}})}{N}$ contains a subbimodule isomorphic with
$$\left(\bigoplus_{i=0}^{r(g)-1} \bim{N}{(\cK_{\omega_g^{i}})}{N}\right) \oplus \left(\bigoplus_{i=1}^{l(g)-1}\bim{N}{(\cK_{ga^{i}g^{-1}})}{N}\right).$$
Since the right dimension of this subbimodule is equal to the right dimension of $\bim{N}{(\cK_g\otimes_N\cK_{g^{-1}})}{N}$, namely $l(g)r(g)$, the two actually coincide.

We are left to prove that the bimodules in the decomposition are irreducible and pairwise nonisomorphic. The irreducibility follows immediately from Lemma \ref{lem.Irred}. Let us now show that the subbimodules are all pairwise nonisomorphic. By Lemma \ref{lem.Irred} and the fact that the $N$-$N$-bimodules $\cK_{\omega_g^{i}}$ are the only $1$-dimensional bimodules in the decomposition, it suffices to show that $\cK_{\omega_g^{i}}$ and $\cK_{\omega_g^{j}}$ are nonisomorphic whenever $i\neq j$. For that, assume the existence of an $N$-$N$-bimodular isomorphism between  $\cK_{\omega_g^{i}}$ and $\cK_{\omega_g^{j}}$. Then we see that there exists a unitary $u\in N$ such that $uxu^* = \alpha_{\omega_g^{i-j}}(x)$ for every $x\in N$. Note then that $u\in N_{r(g)}'\cap N = \C 1$, and so $i=j$. This ends the proof.
\end{proof}

We need one final result before we can start with the proof of the main theorem. For every $g\in\BS(n,m)$ we define $L(g)$ as the nonzero integer satisfying $ga^{L(g)}g^{-1}=a^{r(g)}$. Note that $L(g)\in\{l(g),-l(g)\}$.

\begin{lemma}\label{lem.roots2}
Let $n,m\in\Z$ such that $2\leq n\leq |m|$. Let $(P,\tau)$ be a tracial von Neumann algebra and let $\BS(n,m)\actson P$ be a trace preserving action such that $P\rtimes \langle a^{l(g)}\rangle\subset P\rtimes \BS(n,m)$ is irreducible for every $g\in \BS(n,m)$. Write $M:=P\rtimes \BS(n,m)$ and $N:=P\rtimes \langle a\rangle$. Let $\bim{N}{(\cK_g)}{N}:=\bim{N}{(\overline{\lspan}^{||\cdot||_2} Nu_gN)}{N}$ and $\Omega:=\{e^{2\pi i s/r(g)}\mid s\in \Z, g\in \BS(n,m)\}$. Then for every $\omega,\mu\in\Omega$, we have that
$$\bim{N}{(\cK_\omega \otimes_{N} \cK_g)}{N} \cong \bim{N}{(\cK_g \otimes_{N} \cK_\mu )}{N} \text{ if and only if } \omega^{r(g)}=\mu^{L(g)}.$$
\end{lemma}
\begin{proof} Fix $g\in \BS(n,m)$ and $\omega,\mu\in\Omega$. As before, we define $N_z:=P\rtimes \langle a^{z}\rangle$ for every nonzero integer $z$. 

We first prove the `only if' part of the equivalence. So assume that $\bim{N}{(\cK_\omega \otimes_{N} \cK_g)}{N}$ is isomorphic with $\bim{N}{(\cK_g \otimes_{N} \cK_\mu )}{N}$. For $0\leq i < r(g)$, we denote by $\alpha_{i,\omega}$ the map $\Ad(u_{a^{i}g}^*)\circ \alpha_\omega$ on $N_{r(g)}$ and by $\alpha_{\mu,i}$ the map $\alpha_\mu\circ \Ad(u_{a^{i}g}^*)$ on $N_{r(g)}$. We also define $\bim{N_{r(g)}}{\cH(\alpha_{i,\omega})}{N}$ by $\cH(\alpha_{i,\omega})=L^2(N)$ and $x\xi y = \alpha_{i,\omega}(x)\xi y$. In a similar way, we define $\bim{N_{r(g)}}{\cH(\alpha_{\mu,i})}{N}$. Note that
$$\bim{N_{r(g)}}{(\cK_\omega \otimes_{N} \cK_g)}{N} \cong \bigoplus_{i=0}^{r(g)-1} \bim{N_{r(g)}}{\cH(\alpha_{i,\omega})}{N}$$
and
$$\bim{N_{r(g)}}{(\cK_g \otimes_{N} \cK_\mu )}{N} \cong \bigoplus_{i=0}^{r(g)-1} \bim{N_{r(g)}}{\cH(\alpha_{\mu,i})}{N}.$$
Under these identifications, we have that the set of all $N_{r(g)}$-$N$-bimodular elements of $\B(\cK_\omega \otimes_{N} \cK_g,\cK_g \otimes_{N} \cK_\mu)$ corresponds with
$$B:=\{[x_{i,j}]_{i,j}\mid x_{i,j}\in N \text{ and } x_{i,j}\alpha_{j,\omega}(x)=\alpha_{\mu,i}(x)x_{i,j} \text { for all } x\in N_{r(g)}\}.$$
Take $r\in\{r(h)\mid h\in \BS(n,m)\}$ large enough such that $N_r\subset N_{r(g)}$, $\alpha_{i,\omega}(x)=\Ad(u_{a^{i}g}^*)(x)$ and $\alpha_{\mu,i}(x)=\Ad(u_{a^{i}g}^*)(x)$ for every $x\in N_r$. Then we see that $u_{a^{i}g}x_{i,j}u_{a^jg}^*\in N_r'\cap M=\C 1$, whenever $[x_{i,j}]_{i,j}\in B$. Hence,
$$B=\{[x_{i,j}]_{i,j}\mid x_{i,j}\in \C \delta_{i,j} \text{ and } x_{i,i}\alpha_{i,\omega}(x)=\alpha_{\mu,i}(x)x_{i,i} \text { for all } x\in N_{r(g)}\}.$$
Since $\bim{N}{(\cK_\omega \otimes_{N} \cK_g)}{N} \cong \bim{N}{(\cK_g \otimes_{N} \cK_\mu )}{N}$, there exists a unitary element $[x_{i,j}]_{i,j}$ inside $B$. For this unitary element, we have that $x_{0,0}\in \T 1$ and $x_{0,0}\alpha_{0,\omega}(x)=\alpha_{\mu,0}(x)x_{0,0}$ for every $x\in N_{r(g)}$. From this we get that
$$\alpha_{0,\omega}(x)=\alpha_{\mu,0}(x), \text{ for all } x\in N_{r(g)}.$$
In particular, we have that $\alpha_{0,\omega}(u_{a^{r(g)}})=\alpha_{\mu,0}(u_{a^{r(g)}})$. Now $\alpha_{0,\omega}(u_{a^{r(g)}})=\omega^{r(g)}u_{a^{L(g)}}$, while $\alpha_{\mu,0}(u_{a^{r(g)}})=\mu^{L(g)}u_{a^{L(g)}}$. Hence we have that $\omega^{r(g)}=\mu^{L(g)}$.

Let us now show the `if' part of the equivalence. So assume that $\omega^{r(g)}=\mu^{L(g)}$. Using \cite[Lemma 2.8]{FR12}, we can view $\cK_g$ as an $N$-$N$-subbimodule of $L^2(N)\otimes_{N_{r(g)}} u_gL^2(N)$. Since both bimodules have the same right $N$-dimension, we see that $\cK_g$ and $L^2(N)\otimes_{N_{r(g)}} u_gL^2(N)$ are actually isomorphic. Now define for every normal $*$-homomorphism $\alpha:N_{r(g)}\rightarrow N$ the bimodule $\bim{N_{r(g)}}{\cH(\alpha)}{N}$ by $\cH(\alpha)=L^2(N)$ and $x\xi y= \alpha(x)\xi y$. Then
$$\bim{N_{r(g)}}{u_gL^2(N)}{N} \cong \bim{N_{r(g)}}{\cH(\Ad(u_g^*))}{N}.$$
Hence, $\cK_g$ is isomorphic with $L^2(N)\otimes_{N_{r(g)}} \cH(\Ad(u_g^*))$ as an $N$-$N$-bimodule. But then,
$$\bim{N}{(\cK_\omega \otimes_{N} \cK_g)}{N} \cong \bim{N}{(L^2(N)\otimes_{N_{r(g)}} \cH(\Ad(u_g^*)\circ \alpha_\omega))}{N}$$
and
$$\bim{N}{(\cK_g \otimes_{N} \cK_\mu)}{N} \cong \bim{N}{(L^2(N)\otimes_{N_{r(g)}} \cH(\alpha_\mu \circ \Ad(u_g^*)))}{N}.$$
Since we assumed that $\omega^{r(g)}=\mu^{L(g)}$, we have that $\Ad(u_g^*)\circ \alpha_\omega$ and $\alpha_\mu \circ \Ad(u_g^*)$ coincide on $N_{r(g)}$. Therefore $\bim{N_{r(g)}}{\cH(\Ad(u_g^*)\circ \alpha_\omega)}{N}\cong \bim{N_{r(g)}}{\cH(\alpha_\mu \circ \Ad(u_g^*))}{N}$. So $\bim{N}{(\cK_\omega \otimes_{N} \cK_g)}{N}$ is isomorphic with $\bim{N}{(\cK_g \otimes_{N} \cK_\mu)}{N}$.
\end{proof}

We finally present the proof of the main theorem.

\begin{proof}[Proof of Theorem \ref{Thm.B}]
For $i=1,2$, write $\Gamma_i:=\BS(n_i,m_i)$, $M_i:=P_i\rtimes \Gamma_i$, $N_i:=P_i\rtimes \langle a_i\rangle$ and $N_{i,z}:=P_i\rtimes \langle a_i^z\rangle$ for every nonzero integer $z$. Interchanging if necessary the roles of $M_1$ and $M_2$, there exists a projection $p\in N_1$ and a $*$-isomorphism $\alpha:pM_1p\rightarrow M_2$. By Theorem \ref{thm.Conj}, we may assume that $\alpha(pN_1p)=N_2$.

For $i\in \{1,2\}$ and $g\in\Gamma_i$, we define the $N_i$-$N_i$-bimodule $\cK_g^{i}:=\overline{\lspan}^{||\cdot||_2} N_iu_gN_i$. Recall that $L^2(M_i) = \bigoplus_{\langle a_i\rangle g \langle a_i\rangle} \cK_g^{i}$ as $N_i$-$N_i$-bimodules. Since the isomorphism $\alpha:pM_1p\rightarrow M_2$ satisfies $\alpha(pN_1p)=N_2$, we have that Lemma \ref{lem.Irred} implies that the sets $\{\cK_g^2\mid g\in \Gamma_2\}$ and $\{\alpha(p\cK_g^1p)\mid g\in \Gamma_1\}$ are the same. Looking at the left and right dimensions of the bimodules in both sets, we get that $\{(l(g),r(g))\mid g\in \Gamma_1\}=\{(l(g),r(g))\mid g\in \Gamma_2\}$. 

Note that $n_i=\min(\{l(g)\mid g\in \Gamma_i\}\backslash \{1\})$ and so $n_1=n_2$. On the other hand, $\{l(g)/r(g)\mid g\in\Gamma_i\}=\left( n_i/|m_i| \right)^{\Z}$. Therefore, also $\frac{n_1}{|m_1|}=\frac{n_2}{|m_2|}$. Together, this shows that $n_1=n_2$ and $|m_1|=|m_2|$. It remains to prove that $m_1=m_2$ whenever $n_1\neq |m_1|$. 

Whenever $M$ is a II$_1$ factor and $\cH$ is a nonzero $M$-$M$-bimodule, we write $\Bim_{\cH}(M)$ for the smallest set $\cS$ of isomorphism classes of finite index $M$-$M$-bimodules satisfying the following four conditions:
\begin{enumerate}
\item the isomorphism class of every $M$-$M$-subbimodule of $\cH$ is an element of $\cS$;
\item it is closed under taking Connes tensor products;
\item it is closed under taking $M$-$M$-subbimodules;
\item it is closed under taking contragredients.
\end{enumerate}


Since $\alpha:pM_1p\rightarrow M_2$ satisfies $\alpha(pN_1p)=N_2$, we have that $\alpha$ gives rise to a bijection between $\Bim_{L^2(pM_1p)}(pN_1p)$ and $\Bim_{L^2(M_2)}(N_2)$ preserving Connes tensor products, contragredients and dimensions. On the other hand, we can identify $\Bim_{L^2(M_1)}(N_1)$ with $\Bim_{L^2(pM_1p)}(pN_1p)$ through the map $\cH\rightarrow p\cH p$. Altogether there exists a bijection $\beta$ between $\Bim_{L^2(M_1)}(N_1)$ and $\Bim_{L^2(M_2)}(N_2)$ preserving Connes tensor products, contragredients and dimensions.

We already know that $\beta$ is a bijection between $\{\cK_g^1 \mid g\in \Gamma_1\}$ and $\{\cK_g^2\mid g\in \Gamma_2\}$. Hence, we can choose a map $\sigma:\Gamma_1\rightarrow \Gamma_2$ satisfying $\beta(\cK_g^1)=\cK_{\sigma(g)}^2$ for every $g\in\Gamma_1$. Note that since $\beta$ preserves contragredients, we have that $\cK_{\sigma(g^{-1})}^2=\cK_{\sigma(g)^{-1}}^2$. Also note that $r(g)=r(\sigma(g))$ and $l(g)=l(\sigma(g))$ for every $g\in\Gamma_1$, since $\beta$ preserves left and right dimensions.

Write $\cF:=\{r(g)\mid g\in \Gamma_1\}\setminus \{1\}=\{r(g)\mid g\in \Gamma_2\}\setminus \{1\}$ and define the group $\Omega$ by 
$$\Omega:=\{\omega\in\C \mid \omega^f = 1 \text{ for some } f\in \cF\}.$$ By Lemma \ref{lem.roots}, we have that the group of $1$-dimensional subbimodules of $\{\cK_g^1\otimes_{N_1}\cK_{g^{-1}}^1\mid g\in\Gamma_1\}$ is exactly $\{\cK_{\omega}^1\mid \omega\in\Omega\}$. Similarly, the group of all $1$-dimensional subbimodules of the set $\{\cK_g^2\otimes_{N_2}\cK_{g^{-1}}^2\mid g\in\Gamma_2\}$ is $\{\cK_{\omega}^2\mid \omega\in\Omega\}$. Since $\beta(\{\cK_g^1\otimes_{N_1}\cK_{g^{-1}}^1\mid g\in\Gamma_1\})$ coincides with $\{\cK_g^2\otimes_{N_2}\cK_{g^{-1}}^2\mid g\in\Gamma_2\}$, 
we have that $\beta(\{\cK_{\omega}^1\mid \omega\in\Omega\}) = \{\cK_{\omega}^2\mid \omega\in\Omega\}$. In this way, $\beta$ gives rise to an automorphism $\Delta:\Omega\rightarrow \Omega$ by $\beta(\cK_\omega^1)=\cK_{\Delta(\omega)}^2$.

Now define for $g\in\Gamma_1$ and $h\in\Gamma_2$ the sets
$$W_g^{1}:=\{(\omega,\mu)\in\Omega\times\Omega\mid \cK_\omega^1\otimes_{N_1} \cK_g^1 \cong \cK_g^1 \otimes_{N_1} \cK_\mu^1\}$$
and
$$W_h^{2}:=\{(\omega,\mu)\in\Omega\times\Omega\mid \cK_\omega^2 \otimes_{N_2} \cK_h^2 \cong \cK_h^2 \otimes_{N_2} \cK_\mu^2\}.$$
We have that $(\Delta\times\Delta)(W_g^1)=W_{\sigma(g)}^2$ for every $g\in\Gamma_1$. Using Lemma \ref{lem.roots2} this implies that
\begin{equation}\label{eq.roots2}
(\Delta\times\Delta)(\{(\omega,\mu)\in\Omega\times\Omega\mid \omega^{r(g)}=\mu^{L(g)}\}) = \{(\omega,\mu)\in\Omega\times\Omega\mid \omega^{r(\sigma(g))}=\mu^{L(\sigma(g))}\}.
\end{equation}
Now assume, by way of reaching a contradiction, that $n_1=n_2$ and $m_1=-m_2$ with $n_1\neq |m_1|$. Put $n:=n_1$, $m:=m_1$, $k:=\gcd(n,|m|)$, $n_0:=n/k$ and $m_0:=m/k$. By taking $g\in \Gamma_1$ equal to $b^{-1}$ in (\ref{eq.roots2}), we see that
\[\Delta\times\Delta(\{(\omega,\mu)\in\Omega\times\Omega\mid \omega^{n}=\mu^m\})=\{(\omega,\mu)\in\Omega\times\Omega\mid \omega^{r(\sigma(b^{-1}))}=\mu^{L(\sigma(b^{-1}))}\}.\]
We already know that $r(\sigma(g))=r(g)$ and $l(\sigma(g))=l(g)$ for every $g\in\Gamma_1$. Therefore $r(\sigma(b^{-1}))=n$ and $L(\sigma(b^{-1}))\in \{m,-m\}$. Since $r(h)/L(h)\in (-n/m)^\Z$ for every $h\in \Gamma_2$, we get that $L(\sigma(b^{-1}))$ must be equal to $-m$.

Take $t$ such that $|n_0^tm_0^t|>2$. Define $\omega:=e^{2\pi i/(kn_0^{t+1}m_0^{t})}$ and $\mu:=e^{2\pi i/(kn_0^{t}m_0^{t+1})}$. Then $\omega^n=\mu^m$ and hence
$$\Delta(\mu)^{-m} = \Delta(\omega)^n = \Delta(\omega^n) = \Delta (\mu^m) = \Delta(\mu)^m.$$
Therefore $\Delta(\mu)^{2m}=1$, or equivalently $\mu^{2m}=1$. This is a contradiction, since $|n_0^tm_0^t|>2$. We conclude that $m_1=m_2$ whenever $n_1\neq |m_1|$.
\end{proof}

\section{Two comments on the assumptions of the main theorem}\label{sec.Cond}

In this final section, we examine the assumptions on $\BS(n,m)\actson P$ found in the main theorem. We show that whenever $P$ is abelian, these are equivalent to some seemingly weaker/stronger assumptions.

Throughout this section, let $n$ and $m$ be integers such that $2\leq n \leq |m|$. Let $k$ be the greatest common divisor of $n$ and $|m|$. As before, write $n_0=n/k$, $m_0=m/k$ and $\cF=\{kn_0^s|m_0|^t\mid s,t\in\N, s+t>0\}=\{l(g)\mid g\in \BS(n,m)\}\setminus\{1\}$.

Recall from Lemma \ref{lem.QC} that the quasi-centralizer of $\langle a \rangle$ in $\BS(n,m)$ is $\QC_{\BS(n,m)}(\langle a \rangle) = \{g\in\BS(n,m)\mid ga^{l(g)}g^{-1}=a^{l(g)}\}$. We have the following result.

\begin{lemma}\label{lem.Cond}
Let $\BS(n,m)\actson (X,\mu)$ be a pmp action of $\BS(n,m)$ on a standard probability space $X$. Write $\Gamma:=\BS(n,m)$, $\Lambda:=\QC_{\Gamma}(\langle a \rangle)$, $M:=L^\infty(X)\rtimes \Gamma$ and $N_z:=L^\infty(X)\rtimes \langle a^z\rangle$ for every nonzero integer $zù=$. The following statements are equivalent.
\begin{enumerate}
\item $N_z'\cap M =\C 1$ for every $z\in\cF$.
\item $\Lambda\actson X$ is essentially free and $\langle a^z\rangle\actson X$ is ergodic for every $z\in\cF$.
\item $\Gamma\actson X$ is essentially free and $\langle a^z\rangle\actson X$ is ergodic for every $z\in\cF$.
\end{enumerate}
\end{lemma}
\begin{proof}$1 \Rightarrow 2$. Note that every $\langle a^z\rangle$-invariant element of $L^\infty(X)$ is an element of $N_z'\cap M=\C1$. Therefore $\langle a^z\rangle\actson X$ is ergodic for every $z\in\cF$. It remains to prove that $\Lambda\actson X$ is essentially free. For every $g\in\BS(n,m)$, we write $\Fix(g)$ for the fixed point set of $g$, i.e.\ $\Fix(g):=\{x\in X\mid x=g\cdot x\}$. Assume, by way of reaching a contradiction, that $\Lambda\actson X$ is not essentially free. Then there exists an element $g\in\Lambda\setminus \{e\}$ such that $\mu(\Fix(g))>0$. Since $\Fix(g)$ is an $\langle a^{l(g)}\rangle$-invariant Borel subset of $X$, we get that $\mu(\Fix(g))=1$. To reach a contradiction, observe that $u_g$ is a nontrivial element of $N_{l(g)}'\cap M$.

$2 \Rightarrow 3$. If $\Gamma=\Lambda$, there is clearly nothing to prove. So assume that $\Gamma\neq \Lambda$. Let $g\in\Gamma\setminus\Lambda$ and assume, by way of reaching a contradiction, that $\mu(\Fix(g))>0$. Take a nonzero integer $z$ such that $\mu(a^z\cdot \Fix(g)~\cap~ \Fix(g))>0$. Note that $a^z\cdot \Fix(g)~\cap~ \Fix(g)\subset \Fix(ga^zg^{-1}a^{-z})$. Therefore $\mu(\Fix(ga^zg^{-1}a^{-z}))>0$. On the other hand $ga^zg^{-1}a^{-z}$ belongs to $\Lambda$, since $\Lambda$ is a normal subgroup of $\Gamma$. Furthermore $ga^zg^{-1}a^{-z}$ is nontrivial, since $g$ would otherwise belong to $\cC_{\Gamma}(\langle a^z \rangle)\subset \Lambda$. Altogether we have reached a contradiction.

$3\Rightarrow 1$. Since $\Gamma\actson X$ is essentially free, we have that $L^\infty(X)'\cap M=L^\infty(X)$. Therefore $N_z'\cap M \subset L^\infty(X)$ for every $z\in\cF$. But then, for every $z\in\cF$, we see that $N_z'\cap M$ is the von Neumann algebra of $\langle a^z\rangle$-invariant functions of $L^\infty(X)$. The ergodicity of $\langle a^z\rangle\actson X$ now finishes the proof.
\end{proof}

We also have the following result.

\begin{lemma}\label{lem.Ergo}
Let $\BS(n,m)\actson (X,\mu)$ be a pmp action of $\BS(n,m)$ on a standard probability space $X$. If $\langle a^k\rangle\actson X$ is ergodic, then $\langle a^{z}\rangle \actson X$ is ergodic for every $z\in\cF$.
\end{lemma}  
\begin{proof}
Assume that $\langle a^k\rangle\actson X$ is ergodic. For every $z\in\Z\setminus\{0\}$, we denote by $P(z)$ the $\langle a^z \rangle$-invariant elements of $L^\infty(X)$. By assumption we have that $P(k)=\C 1$. To prove the lemma, we need to show that $P(kn_0^sm_0^t)=\C 1$ for every $s,t\in \N$ with $s+t>0$. So fix $s,t\in \N$ with $s+t>0$ and note that 
\[P(kn_0^sm_0^t)=u_{b^s}^{*}P(km_0^{t+s})u_{b^s} \text{ and } P(kn_0^sm_0^t)=u_{b^t}P(kn_0^{s+t})u_{b^t}^*.\] 
In particular, we have that
$$\dim(P(km_0^{t+s}))=\dim(P(kn_0^sm_0^t))=\dim(P(kn_0^{s+t})).$$ 
Since $n_0$ and $m_0$ are coprime, it suffices to show that $\dim(P(kz))$ divides $z$ whenever $z$ is a nonzero integer. To that end, fix $z\in \Z\setminus\{0\}$ and note that the action $(\langle a^k\rangle/\langle a^{kz}\rangle)\actson P(kz)$ is ergodic since $P(k)=\C 1$. Write $L^\infty(Y,\eta)$ for $P(kz)$ and let $(\langle a^k\rangle/\langle a^{kz}\rangle)\actson Y$ be the ergodic action corresponding to $(\langle a^k\rangle/\langle a^{kz}\rangle)\actson P(kz)$. Then $Y$ is purely atomic. Indeed, if not, then $Y$ would contain a Borel subset $Z$ with $0< \mu(Z) < |z|$. This in turn would mean that
\begin{align*}
1
&=\eta(Y) = \eta((\langle a^k\rangle/\langle a^{kz}\rangle)\cdot Z)\\
&\leq z\; \eta(Z) < 1.
\end{align*}
Let $y\in Y$ be an atom. Then by the orbit-stabilizer theorem, we have that the number of elements in the orbit of $y$ is a divisor of $|\langle a^k\rangle/\langle a^{kz}\rangle|=z$. Since the action is ergodic, the orbit of $y$ is the whole of $Y$. Hence we find that $Y$ consists of exactly $j$ atoms, where $j$ is some divisor of $z$. In other words, $P(kz)$ must be finite dimensional and its dimension should divide $z$. This ends the proof.
\end{proof}

\end{document}